\documentclass[11pt,reqno]{amsart}

\usepackage{amssymb}
\usepackage{fullpage}

\newtheorem{theorem}{Theorem}[section]
\newtheorem{lemma}[theorem]{Lemma}

\newtheorem{proposition}[theorem]{Proposition}

\numberwithin{equation}{section}
\numberwithin{figure}{section}

\newcommand{\re}{\text{{\rm Re}\,}}
\newcommand{\im}{\text{{\rm Im}\,}}
\newcommand{\sgn}{\text{{\rm sgn}\,}}
\newcommand{\R}{\mathbb{R}}
\newcommand{\C}{\mathbb{C}}
\newcommand{\T}{\mathbb{T}}
\newcommand{\E}{\mathbb{E}}
\newcommand{\I}{{\rm i}}

\begin{document}
\setcounter{page}{1}

\thanks{Supported by the grant KAW 2015.0270 from the Knut and Alice Wallenberg
Foundation}

\dedicatory{Dedicated to the memory of Harold Widom 1932--2021}

\title[Strong Szeg\H{o} theorem on a Jordan curve]
{Strong Szeg\H{o} theorem on a Jordan curve}
\author[K.~Johansson]{Kurt Johansson}

\address{
Department of Mathematics,
KTH Royal Institute of Technology,
SE-100 44 Stockholm, Sweden}

\email{kurtj@kth.se}

\begin{abstract}
We consider certain determinants with respect to a sufficiently regular Jordan curve $\gamma$ in the complex plane that generalize Toeplitz determinants which are obtained when the curve is the circle. This also corresponds to studying a planar Coulomb gas on the curve at inverse temperature $\beta=2$. Under suitable assumptions on the curve we prove a strong Szeg\H{o} type asymptotic formula as the size of the determinant grows. The resulting formula involves the Grunsky operator built from the Grunsky coefficients of the exterior mapping function for $\gamma$. As a consequence of our formula we obtain the asymptotics of the partition function for the Coulomb gas on the curve. This formula involves the Fredholm determinant 
of the absolute value squared of the Grunsky operator which equals, up to a multiplicative constant, the Loewner energy of the curve. Based on this we obtain a new characterization of curves with finite Loewner energy called Weil-Petersson quasicircles.

\end{abstract}

\maketitle

\section{Introduction and results}\label{sect1}
\subsection{Definitions and results}\label{sect1.1}
Let $\gamma$ be a Jordan curve in the complex plane and $g:\gamma\mapsto\C$ a given function on the curve. Define the determinant 
\begin{equation}\label{Dneg}
D_n[e^g]=\det\left(\int_\gamma\zeta^j\bar{\zeta}^ke^{g(\zeta)}\,|d\zeta|\right)_{0\le j,k<n}
\end{equation}
assuming that all integrals exist. In the case $\gamma=\T$, the unit circle, this is $(2\pi)^n$ times the Toeplitz determinant with symbol $e^g$.
These determinants are related to orthogonal polynomials on the curve $\gamma$ when the weight function $w=e^g$ is positive, see \cite[Sec. 16.2]{Sz}. 
Note that our definition is different from the one in \cite{Sz}; the $D_n$ there is our $D_{n+1}/L^{n+1}$ where $L$ is the length of the curve. In the case
when $g=0$ these orthogonal polynomials were introduced by Szeg\H{o} in \cite{Sz2}, and some properties of the polynomials and the determinants, like \eqref{Dnquot} below, were investigated.
By Andrieff's identity, we have the integral formula
\begin{equation}\label{intfor}
D_n[e^g]=\frac 1{n!}\int_{\gamma^n}\prod_{1\le \mu\neq\nu\le n}|\zeta_\mu-\zeta_\nu|\prod_{\mu=1}^n e^{g(\zeta_\mu)}\prod_{\mu=1}^n|d\zeta_\nu|.
\end{equation}
Note that $D_n[1]$ is the \emph{partition function} for a planar Coulomb gas
on the curve $\gamma$,
\begin{equation}\label{partfcn}
Z_n(\gamma)=D_n[1]=\frac 1{n!}\int_{\gamma^n}\exp\left[-\sum_{1\le \mu\neq\nu\le n}\log|\zeta_\mu-\zeta_\nu|^{-1}\right]\prod_{\mu=1}^n|d\zeta_\nu|.
\end{equation}
In the case of Toeplitz determinants the strong Szeg\H{o} limit theorem gives a precise asymptotic formula for $D_n[e^g]/D_n[1]$ where $D_n[1]=Z_n(\T)=(2\pi)^n$, see \cite{DeKr}, \cite{Si} for background on, and proofs of, this theorem. In this case the partition function is easy to compute which is not the case for other curves. We want to generalize 
the strong Szeg\H{o} theorem to the case of a more general Jordan curve and also understand the asymptotics of the partition function.
Asymptotic properties of the determinant \eqref{Dneg} were studied in \cite{Sz2} and \cite[Sec. 6.2]{GrSz}, where the asymptotics for a quotient of consecutive determinants was given, see \eqref{Dnquot}. A strong Szeg\H{o} theorem for $D_n[e^g]/D_n[1]$ was proved in \cite[Sec. III]{Jo1}.
In this paper we prove a precise asymptotic formula for $D_n[e^g]$ under a somewhat weaker, but certainly not optimal, condition on the curve $\gamma$, and our
assumption on the function $g$ is optimal.
See Section \ref{sect1.2} below for further comments and background. The asymptotics of the partition function $Z_n(\gamma)$ turns out to be interesting and we will discuss its asymptotics under optimal conditions in Section \ref{sec:WPZn}. We will not prove any results for the $\beta$-ensemble corresponding to \eqref{intfor}, but we give a heuristic 
discussion in Subsection \ref{subsec:beta} and Section \ref{sec:beta}.

We can expect that the leading order asymptotics of $D_n[e^g]$ should be given by $\exp(-n^2V(\gamma))$, where $V(\gamma)$ is the \emph{logarithmic energy} of $\gamma$. The logarithmic energy is defined by 
\begin{equation*}
V(\gamma)=\inf_\mu\int_\gamma\int_\gamma\log|\zeta_1-\zeta_2|^{-1}\,d\mu(\zeta_1)d\mu(\zeta_2),
\end{equation*}
where the infimum is over all probability measures $\mu$ on $\gamma$. The \emph{logarithmic capacity} of $\gamma$ can be defined by $\text{\rm cap}(\gamma)=\exp(-V(\gamma))$.
Let $\Omega$ be the unbounded component of the complement of $\gamma$ and let $\text{\rm cap}(\gamma)\phi:\{z\,;\,|z|>1\}\mapsto\Omega$ be the exterior mapping function with the expansion,
\begin{equation}\label{phiexp}
\phi(z)=z+\phi_0+\phi_{-1}z^{-1}+\dots
\end{equation}
around infinity. If $|z|> 1, |\zeta|>1$, we have the expansion
\begin{equation}\label{logphiq}
\log\frac{\phi(\zeta)-\phi(z)}{\zeta-z}=-\sum_{k,\ell=1}^\infty a_{k\ell}\zeta^{-k}z^{-\ell},
\end{equation}
where $a_{k\ell}=a_{\ell k}\in\C$ are the \emph{Grunsky coefficients}, see e.g. \cite[Sec. 3.1]{Po}. If $\gamma$ is a quasicircle, i.e. it is the image of the unit circle under a quasiconformal mapping of the plane, there is a constant $\kappa<1$ such that
\begin{equation}\label{Grunsky1}
\sum_{k=1}^\infty \left|\sum_{\ell=1}^\infty \sqrt{k\ell}\,a_{k\ell}w_\ell\right|^2\le\kappa^2\sum_{k=1}^\infty |w_k|^2,
\end{equation}
and
\begin{equation}\label{Grunsky2}
\left|\sum_{k,\ell=1}^\infty\sqrt{k\ell}\,a_{k\ell}w_kw_\ell\right|\le\kappa\sum_{k=1}^\infty |w_k|^2,
\end{equation}
called the \emph{Grunsky inequalities}, see \cite[Sec. 9.4]{Po}.
Write
\begin{equation}\label{bkl}
b_{k\ell}=\sqrt{k\ell}\,a_{k\ell}=b_{k\ell}^{(1)}+\I b_{k\ell}^{(2)},
\end{equation}
where $b_{k\ell}^{(j)}\in\R$ and $\I=\sqrt{-1}$. Consider the \emph{Grunsky operator} $B$ and its real and imaginary parts,
\begin{equation}\label{Bop}
B=(b_{k\ell})_{k,\ell\ge 1}, \quad B^{(j)}=(b_{k\ell}^{(j)})_{k,\ell\ge 1}, j=1,2,
\end{equation}
which are bounded operators on $\ell^2(\C)$ by (\ref{Grunsky1}) with norm $\le\kappa<1$. Define the operator $K$ on $\ell^2(\C)\oplus\ell^2(\C)$, by
\begin{equation*}
K=\begin{pmatrix} B^{(1)} & B^{(2)} \\ B^{(2)} & -B^{(1)}  \end{pmatrix}.
\end{equation*}
Note that $K$ is real and symmetric. 

Expand the function $g(\phi(e^{\I\theta}))$ in a Fourier series,
\begin{equation}\label{gFour}
g(\phi(e^{\I\theta}))=\frac {a_0}2+\sum_{k=1}^\infty a_k\cos k\theta+b_k\sin k\theta,
\end{equation}
where $a_k,b_k\in\C$ and we assume that $g(\phi(e^{\I\theta}))$ is integrable. Define the infinite column vector in $\ell^2(\C)\oplus\ell^2(\C)$,
\begin{equation}\label{gcol}
\mathbf{g}=\begin{pmatrix} (\frac 12\sqrt{k}a_k)_{k\ge 1} \\ (\frac 12 \sqrt{k}b_k)_{k\ge 1} \end{pmatrix}.
\end{equation}
We can now state our main theorem which gives the asymptotics for the determinant \eqref{Dneg} as $n\to\infty$.

\begin{theorem}\label{thm:main}
Assume that the Jordan curve $\gamma$ is $C^{5+\alpha}$, $\alpha>0$, and that 
\begin{equation}\label{gcond}
\sum_{k=1}^\infty k(|a_k|^2+|b_k|^2)<\infty.
\end{equation}
We then have the asymptotic formula
\begin{equation}\label{Dnas}
D_n[e^g]=\frac{(2\pi)^n\text{\rm cap}(\gamma)^{n^2}}{\sqrt{\det(I+K)}}\exp\left(na_0/2+\mathbf{g}^t(I+K)^{-1}\mathbf{g}+o(1)\right),
\end{equation}
as $n\to\infty$.
\end{theorem}

The theorem will be proved in the next section. We see that the geometry of the curve enters via the operator $K$, and also directly via $\phi$ since we have the composition
of $g$ with $\phi$ in \eqref{gFour}. If $\gamma$ is the unit circle, we have $K=0$, $\text{\rm cap}(\gamma)=1$, and we get the usual strong Szeg\H{o} limit theorem. 
For the partition function (\ref{partfcn}) we obtain the following asymptotic formula,
\begin{equation}\label{partfcnformula}
Z_n(\gamma)=D_n[1]=\exp\left(n^2\log \text{\rm cap}(\gamma)+n\log 2\pi-\frac 12\log\det(I+K)+o(1)\right),
\end{equation}
as $n\to\infty$.

\subsection{Discussion}\label{sect1.2}
It was proved in \cite[Sec. 6.2]{GrSz} that if $\gamma$ is an analytic curve then
\begin{equation}\label{Dnquot}
\lim_{n\to\infty}\text{\rm cap}(\gamma)^{-2n-1}\frac{D_{n+1}[e^g]}{D_n[e^g]}=2\pi\exp\left(\frac 1{2\pi}\int_{-\pi}^\pi g(\phi(e^{\I\theta}))\,d\theta\right),
\end{equation}
which is a consequence of (\ref{Dnas}) above. In \cite[Sec. III]{Jo1} the following relative Szeg\H{o} type theorem was proved. If $g$ is $C^{1+\alpha}$ on 
$\gamma$ and $\gamma$ is $C^{10+\alpha}$ for some $\alpha>0$, then
\begin{equation}\label{Dnasold}
\frac{D_n[e^g]}{D_n[1]}=\exp[na_0/2+\mathbf{g}^t(I+K)^{-1}\mathbf{g}+o(1)]
\end{equation}
as $n\to\infty$. The expression for the constant term in the exponent in the right side was less clear in \cite{Jo1}, see Theorem 7.1 and its Corollary. The form give here
for the constant term is more elegant and satisfactory. A calculation shows that they are identical. We see that Theorem \ref{thm:main} is a strengthening of the earlier result. In 
particular it is not a relative Szeg\H{o} theorem since we do not divide by $D_n[1]$, and hence we also get asymptotics for the partition function as in \eqref{partfcnformula}.
The condition \eqref{gcond} in the theorem on $g$ is natural since $(I+K)^{-1}$ is an operator on $\ell^2(\C)$ we have to require that $\mathbf{g}\in\ell^2(\C)$ which is
exactly \eqref{gcond}. Thus the condition on $g$ is optimal. The condition in Theorem \ref{thm:main} that $\gamma$ is $C^{5+\alpha}$ is certainly not optimal although it is not immediately clear what the optimal condition is. If we consider the case when $g=0$ we can say more about the optimal condition on $\gamma$. This is the topic of the next subsection. At the end of that subsection we give a conjecture on the the optimal condition on $\gamma$ in the theorem.

If $\gamma$ has a cusp it is not a quasicircle and the Grunsky inequality \eqref{Grunsky1} no longer holds with $\kappa<1$. It would be interesting to see what the effect of a cusp is on the asymptotics of the determinant. Another question is to generalize the result to an arc instead of a Jordan curve. There are results in
case when $\gamma$ is an interval in $\R$ since in that case we get a Hankel determinant, see \cite{Hi}, \cite{Ge}, \cite{Jo1}. In the case of an arc on the unit circle there is
an asymptotic formula due to Widom in \cite{Wi}. See also \cite{Kr} when we have Fisher-Hartwig singularities, and \cite{DuKo} for a relative Szeg\H{o} theorem.

\subsection{Convergence of the partition function and Weil-Petersson quasicircles}\label{sec:WPZn}
Note that since $\det(I+K)=\det(I-B^*B)$, see \eqref{detid}, it follows that \eqref{partfcnformula} can be written as
\begin{equation}\label{logdetlimit}
\lim_{n\to\infty}\log\frac{Z_n(\gamma)/\text{\rm cap}(\gamma)^{n^2}}{Z_n(\T)/\text{\rm cap}(\T)^{n^2}}=
\lim_{n\to\infty}\log \frac{Z_n(\gamma)}{(2\pi)^n\text{\rm cap}(\gamma)^{n^2}}=-\frac 12\log\det(I-B^*B),
\end{equation}
since $\text{\rm cap}(\T)=1$.
Interestingly, the right side of \eqref{logdetlimit} has occurred in other contexts. In \cite{TT} it appears, up to a multiplicative constant, 
under the name universal Liouville action which is a K\"ahler potential for the Weil-Petersson metric on the $T_0(1)$ component of the universal Teichm\"uller space $T(1)$. It is also the so called \emph{Loewner energy} of the Jordan curve, see \cite{RoWa}, \cite{Wa}. Motivated in part by connections to the Schramm-Loewner Evolution, the Loewner energy has been further studied in \cite{ViWa1}, \cite{ViWa2}. Curves with the property that the Grunsky operator is a Hilbert-Schmidt operator are called \emph{Weil-Petersson quasicircles}. There are many different characterizations and possible definitions of Weil-Petersson quasicircles, see \cite{Bi} and \cite[Sect. 8]{Wa}. It is therefore natural to 
conjecture that \eqref{logdetlimit} holds if and only if $\gamma$ is a Weil-Petersson quasicircle. We will prove this but in order to state a theorem let us be a bit more precise.

Let $\gamma$ be a Jordan curve and let $\phi(z)$ be the exterior mapping function for $\gamma$ as above.
Set $\phi_r(z)=\frac 1r\phi(rz)$ for $|z|\ge\rho^{-1}$, where $1<\rho<r$, and let $\gamma_r$ be the image of $\T$ under $\phi_r$. Note that $\phi_r$ is an analytic curve so in particular 
$Z_n(\gamma_r)$ is well-defined. Define the function
\begin{equation}\label{Enr}
E_n(r)=\log\frac{Z_n(\gamma_r)}{(2\pi)^n\text{\rm cap}(\gamma)^{n^2}},
\end{equation}
which we informally can think of as a finite $n$ Loewner energy of $\gamma_r$. As $n\to\infty$ it converges to a multiple of the Loewner energy by \eqref{logdetlimit}.
The following lemma will be proved in Section \ref{sec:WPproof}.

\begin{lemma}\label{lem:Enrdecr}
The function $E_n(r)$ is decreasing in $(1,\infty)$ for every $n\ge 1$.
\end{lemma}

The sequence $E_n(r)$ is also increasing in $n$ for each fixed $r>1$. This is the content of Lemma \ref{lem:Enrincr} below.

The determinant in \eqref{Dneg} is not well-defined for a general Jordan curve. 
Therefore we define the $n$:th partition function for the Jordan curve $\gamma$ by
\begin{equation}\label{Zndef}
Z_n(\gamma)=\lim_{r\to 1+}Z_n(\gamma_r)
\end{equation}
with value in $\R\cup\{\infty\}$. The limit exists by Lemma \ref{lem:Enrdecr} possibly equal to infinity. We can now formulate our theorem which gives a new characterization of 
Weil-Petersson quasicircles and a way to compute their Loewner energy.

\begin{theorem}\label{thm:WP}
The Jordan curve $\gamma$ is a Weil-Petersson quasicircle if and only if
\begin{equation}\label{ZnWPbound}
\limsup_{n\to\infty}\frac{Z_n(\gamma)}{(2\pi)^n\text{\rm cap}(\gamma)^{n^2}}<\infty
\end{equation}
and in that case we have the limit \eqref{logdetlimit}.
\end{theorem}
In fact, by Lemma \ref{lem:Enrincr}, the sequence in \eqref{ZnWPbound} is increasing in $n$ so we could replace the upper limit with a proper limit.
The theorem is proved in Section \ref{sec:WPproof}. In view of this theorem it is reasonable to conjecture that the optimal condition on the curve $\gamma$ in theorem \ref{thm:main} is that $\gamma$ is a Weil-Petersson quasicircle.

\subsection{The $\beta$-ensemble}\label{subsec:beta}
Consider the $\beta$-ensemble corresponding to \eqref{intfor}, i.e. consider
\begin{equation*}
D_{n,\beta}[e^g]=\frac 1{n!}\int_{\gamma^n}\prod_{1\le \mu\neq\nu\le n}|\zeta_\mu-\zeta_\nu|^{\beta/2}\prod_{\mu=1}^n e^{g(\zeta_\mu)}\prod_{\mu=1}^n|d\zeta_\nu|,
\end{equation*}
where $\beta>0$. This is not a determinant when $\beta\neq 2$ but is the quantity analogous to \eqref{Dneg}. The corresponding partition function is $Z_{n,\beta}(\gamma)=D_{n,\beta}[1]$. If $\gamma=\T$, then
\begin{equation*}
Z_{n,\beta}(\T)=\frac{(2\pi)^n}{n!}\frac{\Gamma(1+\beta n/2)}{\Gamma(1+\beta/2)^n}.
\end{equation*}
The expansion (\ref{logphiq}) gives
\begin{equation}\label{logphiprime}
\log\phi'(z)=-\sum_{k=2}^\infty\left(\sum_{j=1}^{k-1} a_{j,k-j}\right)z^{-k}.
\end{equation}
Define
\begin{equation*}
\mathbf{g}_{\beta}=(\beta/2-1)\mathbf{d}+\mathbf{g},
\end{equation*}
where
\begin{equation*}
\mathbf{d}=\begin{pmatrix} \left(\frac 12\sqrt{k}\re\left(\sum_{j=1}^{k-1}a_{j,k-j}\right)\right)_{k\ge 1} \\ \left(\frac 12\sqrt{k}\im\left(\sum_{j=1}^{k-1}a_{j,k-j}\right)\right)_{k\ge 1}
\end{pmatrix}
\end{equation*}
comes from $\log|\phi'(z)|$.
Here the $k=1$ component is $=0$, compare with \eqref{logphiprime}.
We \emph{conjecture} that, if $a_0=0$, then
\begin{equation}\label{betalimit}
\lim_{n\to\infty}\frac{D_{n,\beta}[e^g]}{Z_{n,\beta}(\T)\text{\rm cap}(\gamma)^{\beta n(n-1)/2+n}}=\frac 1{\sqrt{\det(I+K)}}
\exp\left(\frac 2{\beta}\mathbf{g}_{\beta}^t(I+K)^{-1}\mathbf{g}_{\beta}\right).
\end{equation}
We will give a heuristic argument for this in Section \ref{sec:beta}. If we let $\beta=2$ in this argument it also gives the idea behind the proof of \eqref{Dnas} given in the paper, although making it precise is more complicated. If the conjecture is correct, we see that the appearance of the Fredholm determinant $\det(I+K)=\det(I-B^*B)$ is not
related to $\beta=2$. The proof of \eqref{Dnasold} in \cite{Jo1} does not use the fact that $\beta=2$ in an essential way, so by slightly modifying that proof it should be possible
to show that
\begin{equation*}
\lim_{n\to\infty}\frac{D_{n,\beta}[e^g]}{D_{n,\beta}[1]}=\exp\left(\frac 2{\beta}\mathbf{g}^t(I+K)^{-1}\mathbf{g}+2(1-\frac 2{\beta})\mathbf{d}^t(I+K)^{-1}\mathbf{g}\right),
\end{equation*}
under the assumptions in \cite{Jo1}.

\subsection*{Acknowledgements} I thank Fredrik Viklund for many useful comments on the paper and interesting discussions, and for pointing out the monotonicity in $r$ of the
Loewner energy for $\gamma_r$ which motivated Lemma \ref{lem:Enrdecr}.
Thanks also to Klara Courteaut and Gaultier Lambert for comments on the exposition.

\section{Proof of the main theorem}\label{sect2}
In this section we will prove theorem \ref{thm:main}. An essential ingredient is Lemma \ref{lem:Gmnres} which is proved in the next section.
Assume that $\gamma$ is a $C^{5+\alpha}$-curve for some $\alpha>0$. Then, by a theorem of Kellogg, 
see e.g. \cite[Thm. II 4.3]{GM}, the map $\phi$ can be extended to $|z|\ge 1$ in such a way that $\phi$ is $C^{5+\alpha}$ on $\T$, and $\phi'(z)\neq 0$ when $|z|\ge 1$. Clearly, the expansion \eqref{logphiq} then holds for $|z|, |\zeta|\ge 1$. 
It follows from (\ref{intfor}), by introducing the parametrization $\zeta_\mu=\text{\rm cap}(\gamma)\phi(e^{\I\theta_\mu})$, that
\begin{equation}\label{parintfor}
D_n[e^g]=\frac{ \text{\rm cap}(\gamma)^{n^2}}{n!}\int_{[-\pi,\pi]^n}\exp\left[\sum_{\mu\neq\nu}\log|\phi(e^{\I\theta_\mu})-\phi(e^{\I\theta_\nu})|+\sum_\mu\log|\phi'(e^{\I\theta_\mu})|
+g(\phi(e^{\I\theta_\mu}))\right]\,d\theta.
\end{equation}
Combining \eqref{logphiprime} with (\ref{logphiq}) leads to the identity
\begin{align}\label{logfor}
&\sum_{\mu\neq\nu}\log|\phi(e^{\I\theta_\mu})-\phi(e^{\I\theta_\nu})|+\sum_\mu\log|\phi'(e^{\I\theta_\mu})|\\
&=\sum_{\mu\neq\nu}\log|e^{\I\theta_\mu}-e^{\I\theta_\nu}|-\re\sum_{k,\ell=1}^\infty a_{k\ell}\left(\sum_\mu e^{-\I k\theta_\mu}\right)\left(\sum_\nu e^{-\I\ell\theta_\nu}\right).\notag
\end{align}
Let
\begin{equation}\label{Enexp}
\E_n[(\cdot)]=\frac 1{(2\pi)^nn!}\int_{[-\pi,\pi]^n}\exp\bigg(\sum_{\mu\neq\nu}\log|e^{\I\theta_\mu}-e^{\I\theta_\nu}|\bigg)(\cdot)\,d\theta,
\end{equation}
be the expectation over the eigenvalues $e^{\I\theta_\mu}$ of a random unitary matrix with respect to normalized Haar measure, \cite{Mec}.
It follows from (\ref{parintfor}), (\ref{logfor}) and (\ref{Enexp}) that
\begin{equation}\label{Dn1}
D_n[e^g]=(2\pi)^n\text{\rm cap}(\gamma)^{n^2}\E_n\left[\exp\left(-\re\sum_{k,\ell=1}^\infty a_{k\ell}\left(\sum_\mu e^{-\I k\theta_\mu}\right)\left(\sum_\nu e^{-\I\ell\theta_\nu}\right)+
\sum_\mu g(\phi(e^{\I\theta_\mu}))\right)\right].
\end{equation}

To proceed we will need some linear algebra. Let $b_{k\ell}$ be given by \eqref{bkl} and define
\begin{equation}\label{Bm}
B_m=(b_{k\ell})_{1\le k,\ell\le m}=B_m^{(1)}+\I B_m^{(2)},
\end{equation}
so that $B_m^{(1)}$ and $B_m^{(2)}$ are real symmetric $m$ times $m$ matrices.  Since $B_m$ is a complex, symmetric matrix there is a unitary matrix $U_m=R_m+\I S_m$, with $R_m$ and $S_m$ real, such that 
\begin{equation}\label{Bmu}
B_m=U_m\Lambda_m U_m^t,
\end{equation}
where $\Lambda_m=\text{\rm diag}(\lambda_{m,1},\dots,\lambda_{m,m})$ and $\lambda_{m,k}$, $1\le k\le m$, are the singular values of $B_m$, \cite[Sec. 4.4]{HoJo}. Define the real, symmetric $2m$ by $2m$ matrix
\begin{equation*}
K_m=\begin{pmatrix} B_m^{(1)} & B_m^{(2)} \\ B_m^{(2)} & -B_m^{(1)} \end{pmatrix},
\end{equation*}
and the matrices
\begin{equation*}
T_m=\begin{pmatrix} R_m & S_m \\ S_m & -R_m\end{pmatrix}, \quad 
\tilde{\Lambda}_m=\begin{pmatrix} \Lambda_m & 0 \\ 0 & -\Lambda_m \end{pmatrix}.
\end{equation*}

\begin{lemma}\label{lem:linalg}
The matrix $T_m$ is orthogonal and
\begin{equation}\label{KTlambda}
K_m=T_m\tilde{\Lambda}_mT_m^t,
\end{equation}
so that $K_m$ has eigenvalues $\pm\lambda_{m,k}$.  These eigenvalues satisfy
\begin{equation}\label{eigenineq}
|\lambda_{m,k}|\le\kappa <1,
\end{equation}
where $\kappa$ is the constant in the Grunsky inequality (\ref{Grunsky2}). Furthermore, if $\mathbf{x}=(x_j)_{1\le j\le m}$, $\mathbf{y}=(y_j)_{1\le j\le m}$ are real column vectors , then
\begin{equation}\label{Rebkl}
\re\sum_{k,\ell=1}^m b_{k\ell}(x_k-\I y_k)(x_\ell-\I y_\ell)=\begin{pmatrix} \mathbf{x} \\ \mathbf{y} \end{pmatrix}^t K_m \begin{pmatrix} \mathbf{x} \\ \mathbf{y} \end{pmatrix}.
\end{equation}
\end{lemma}

\begin{proof}
That $T_m$ is orthogonal follows from $U_mU_m^*=I$. The identities (\ref{Bm}), (\ref{Bmu}) and $U_m=R_m+\I S_m$ give
\begin{align*}
B_m^{(1)}&=R_m\Lambda_m R_m^t-S_m\Lambda_m S_m^t,\\
B_m^{(2)}&=R_m\Lambda_m S_m^t+S_m\Lambda_m R_m^t,
\end{align*}
which translates into (\ref{KTlambda}). We also see that
\begin{align*}
\re\sum_{k,\ell=1}^m b_{k\ell}(x_k-\I y_k)(x_\ell-\I y_\ell)&=\sum_{k,\ell=1}^mb_{k\ell}^{(1)}(x_kx_\ell-y_ky_\ell)+b_{k\ell}^{(2)}(x_ky_\ell+y_kx_\ell)\\
&=\begin{pmatrix} \mathbf{x} \\ \mathbf{y} \end{pmatrix}^t \begin{pmatrix} B_m^{(1)} & B_m^{(2)} \\ B_m^{(2)} & -B_m^{(1)} \end{pmatrix}
 \begin{pmatrix} \mathbf{x} \\ \mathbf{y} \end{pmatrix},
 \end{align*}
 which proves (\ref{Rebkl}). It follows from (\ref{Grunsky2}) and (\ref{Rebkl}) that
 \begin{equation*}
 \left|\begin{pmatrix} \mathbf{x} \\ \mathbf{y} \end{pmatrix}^t K_m \begin{pmatrix} \mathbf{x} \\ \mathbf{y} \end{pmatrix}\right|
 \le\kappa \begin{pmatrix} \mathbf{x} \\ \mathbf{y} \end{pmatrix}^t \begin{pmatrix} \mathbf{x} \\ \mathbf{y} \end{pmatrix},
 \end{equation*}
 which shows that all the eigenvalues of $K_m$ have absolute value $\le\kappa$.

\end{proof}

Let
\begin{equation*}
\mathbf{X}=\left(\frac 1{\sqrt{k}}\sum_\mu \cos k\theta_\mu\right)_{k\ge 1},\quad \mathbf{Y}=\left(\frac 1{\sqrt{k}}\sum_\mu \sin k\theta_\mu\right)_{k\ge 1},
\end{equation*}
be infinite column vectors, and let $P_m$ denote projection onto the first $m$ components. It follows from (\ref{KTlambda}) and (\ref{Rebkl}) that
\begin{equation}\label{Reakl}
-\re\sum_{k,\ell=1}^ma_{k\ell}\left(\sum_\mu e^{-\I k\theta_\mu}\right)\left(\sum_\nu e^{-\I \ell\theta_\nu}\right)=
\begin{pmatrix}P_m\mathbf{X} \\P_m\mathbf{Y} \end{pmatrix}^tT_m \begin{pmatrix} -\Lambda_m & 0\\0 & \Lambda_m \end{pmatrix}T_m^t
\begin{pmatrix}P_m\mathbf{X} \\P_m\mathbf{Y} \end{pmatrix}.
\end{equation}
Define, for $\zeta\in\C$,
\begin{equation}\label{Mmzeta}
M_m(\zeta)=\begin{pmatrix} \zeta\Lambda_m^{1/2} & 0\\0 & \Lambda_m^{1/2} \end{pmatrix}T_m^t
\begin{pmatrix}P_m\mathbf{X} \\P_m\mathbf{Y} \end{pmatrix},
\end{equation}
so that
\begin{equation}\label{ReMM}
-\re\sum_{k,\ell=1}^ma_{k\ell}\left(\sum_\mu e^{-\I k\theta_\mu}\right)\left(\sum_\nu e^{-\I \ell\theta_\nu}\right)=M_m(\I)^tM_m(\I).
\end{equation}
Without loss of generality we can assume that $a_0=0$ in (\ref{gFour}) by just subtracting off the mean. For $\zeta\in\C$, we define
\begin{equation}\label{akzeta}
a_k(\zeta)=\frac 12\sqrt{k}(\re a_k+\zeta\im a_k), \quad b_k(\zeta)=\frac 12\sqrt{k}(\re b_k+\zeta\im b_k),
\end{equation}
and the infinite column vectors
\begin{equation*}
\mathbf{a}(\zeta)=(a_k(\zeta))_{k\ge 1}, \quad \mathbf{b}(\zeta)=(b_k(\zeta))_{k\ge 1},
\end{equation*}
which lie in $\ell^2(\C)$ by the assumption (\ref{gcond}). Set
\begin{equation*}
\mathbf{g}(\zeta)=\begin{pmatrix} \mathbf{a}(\zeta)\\\mathbf{b}(\zeta)\end{pmatrix},
\end{equation*}
so that $\mathbf{g}(\I)=\mathbf{g}$ given by \eqref{gcol}.
We see that \eqref{gFour} can be written
\begin{equation}\label{gXY}
\sum_\mu g(\phi(e^{\I\theta_\mu}))=2\mathbf{g}(\I)^t\begin{pmatrix}\mathbf{X} \\\mathbf{Y} \end{pmatrix}.
\end{equation}
Define
\begin{equation}\label{wm}
w_m(\theta_1,\theta_2)=\re\sum_{k\vee\ell>m}a_{k\ell}e^{-\I k\theta_1-\I \ell\theta_2},
\end{equation}
and
\begin{equation}\label{Wm}
W_m(\theta)=-\sum_{\mu,\nu}w_m(\theta_\mu,\theta_\nu).
\end{equation}
Combining (\ref{ReMM}), (\ref{gXY}), (\ref{wm}) and (\ref{Wm}), we obtain the identity
\begin{equation*}
-\re\sum_{k,\ell=1}^\infty a_{k\ell}\left(\sum_\mu e^{-\I k\theta_\mu}\right)\left(\sum_\nu e^{-\I \ell\theta_\nu}\right)+\sum_{\mu}g(\phi(e^{\I \theta_\mu}))=
M_m(\I)^tM_m(\I)+2\mathbf{g}(\I)^t\begin{pmatrix}\mathbf{X} \\\mathbf{Y} \end{pmatrix}+W_m(\theta),
\end{equation*}
for every $m\ge 1$.
Using this identity in (\ref{Dn1}) leads us to define the entire function
\begin{equation}\label{Gmnzeta}
G_{m,n}(\zeta)=\E_n[\exp(M_m(\zeta)^tM_m(\zeta)+2\mathbf{g}(\zeta)\begin{pmatrix}\mathbf{X} \\\mathbf{Y} \end{pmatrix}+W_m(\theta))],
\end{equation}
for $\zeta\in\C$, so that, for any $m\ge 1$,
\begin{equation}\label{Dn2}
D_n(e^g)=(2\pi)^n\text{\rm cap}(\gamma)^{n^2}G_{m,n}(\I).
\end{equation}
Note that $G_{m,n}(\I)$ is independent of $m$, but of course for other $\zeta$ the function $G_{m,n}(\zeta)$ does depend on $m$.

The expression $M_m(\zeta)^tM_m(\zeta)$ is a quadratic form in $\mathbf{X}$ and $\mathbf{Y}$ and we want instead to have a linear form in
$\mathbf{X}$ and $\mathbf{Y}$. This can be achieved by using a Gaussian integral, an idea that was also used in \cite[Sec. 6.5]{JoLa}.
Let $\mathbf{u}=(u_k)_{1\le k\le m}$, $\mathbf{v}=(v_k)_{1\le k\le m}$ be real column vectors. Then,
\begin{equation*}
\exp(M_m(\zeta)^tM_m(\zeta))=\frac 1{\pi^m}\int_{\R^m}du\int_{\R^m}dv\exp\left(-\begin{pmatrix} \mathbf{u} \\ \mathbf{v}\end{pmatrix}^t\begin{pmatrix} \mathbf{u} \\ \mathbf{v}\end{pmatrix}
+2 \begin{pmatrix} \mathbf{u} \\ \mathbf{v}\end{pmatrix}^tM_m(\zeta)\right)
\end{equation*}
and Fubini's theorem in (\ref{Gmnzeta}) to get the formula
\begin{equation}\label{Gmnzeta2}
G_{m,n}(\zeta)=\frac 1{\pi^m}\int_{\R^m}du\int_{\R^m}dv\exp(-\begin{pmatrix} \mathbf{u} \\ \mathbf{v}\end{pmatrix}^t\begin{pmatrix} \mathbf{u} \\ \mathbf{v}\end{pmatrix})
\E_n[\exp(2 \begin{pmatrix} \mathbf{u} \\ \mathbf{v}\end{pmatrix}^tM_m(\zeta)+2\mathbf{g}(\zeta)\begin{pmatrix}\mathbf{X} \\\mathbf{Y} \end{pmatrix}+W_m(\theta))].
\end{equation}
From the definition (\ref{Mmzeta}) we see that
\begin{equation*}
\begin{pmatrix} \mathbf{u} \\ \mathbf{v}\end{pmatrix}^tM_m(\zeta)=L_m(\zeta)^t\begin{pmatrix}\mathbf{X} \\\mathbf{Y} \end{pmatrix},
\end{equation*}
where
\begin{equation}\label{Lmzeta}
L_m(\zeta)=\begin{pmatrix} P_m & 0 \\ 0 & P_m \end{pmatrix}^tT_m\begin{pmatrix} \zeta\Lambda_m^{1/2} & 0 \\ 0  & \Lambda_m^{1/2} \end{pmatrix}
\begin{pmatrix} \mathbf{u} \\ \mathbf{v}\end{pmatrix}.
\end{equation}
Thus,
\begin{equation}\label{Gmnzeta3}
G_{m,n}(\zeta)=\frac 1{\pi^m}\int_{\R^m}du\int_{\R^m}dv\exp(-\begin{pmatrix} \mathbf{u} \\ \mathbf{v}\end{pmatrix}^t\begin{pmatrix} \mathbf{u} \\ \mathbf{v}\end{pmatrix})
\E_n[\exp(2(L_m(\zeta)+\mathbf{g}(\zeta))^t\begin{pmatrix}\mathbf{X} \\\mathbf{Y} \end{pmatrix}+W_m(\theta))].
\end{equation}
Note that by the definitions
\begin{equation}\label{Gmnineq}
|G_{m,n}(\zeta)|\le G_{m,n}(\re\zeta).
\end{equation}
We will now state a lemma that will allow us to prove the theorem. The proof of the lemma will be given in the next section.
Define the function
\begin{equation}\label{fzetalambda}
f_\zeta(\lambda)=\begin{cases} (1-\zeta^2\lambda)^{-1}\zeta^2\lambda, &\lambda\ge 0 \\ -(1+\lambda)^{-1}\lambda, &\lambda<0 \end{cases}
\end{equation}
on the real line. We can use spectral calculus to define $f_\zeta(K)$. Recall that $K$ is a symmetric trace class operator with spectrum in $[-\kappa,\kappa]$. Define
\begin{equation}\label{Gzeta}
G(\zeta)=\frac 1{\sqrt{\det(I-\zeta^2|B|)\det(I-|B|)}}\exp\left(\mathbf{g}(\zeta)^t(I+f_\zeta(K))\mathbf{g}(\zeta)\right),
\end{equation}
which is holomorphic in $|\zeta|<\kappa$. Note that $B$ is a trace-class operator by Lemma \ref{lem:tc} below.

\begin{lemma}\label{lem:Gmnres}
Let $G_{m,n}(\zeta)$ be defined by (\ref{Gmnzeta}). Then if $\rho\in(1,1/\sqrt{\kappa})$ there is a constant $C$ so that
\begin{equation}\label{Gmnest}
|G_{m,n}(\zeta)|\le C
\end{equation}
for all $|\zeta|\le\rho$, $n\ge 1$ and $m$ sufficiently large.
Also, if $\zeta$ is real and $|\zeta|\le\rho$, then
\begin{equation}\label{Gmnlimit}
\lim_{m\to\infty}\lim_{n\to\infty}G_{m,n}(\zeta)=G(\zeta).
\end{equation}
\end{lemma}

Assume Lemma \ref{lem:Gmnres}. It follows from (\ref{Gmnest}) that $\{G_{m,n}(\zeta)\}$ is a normal family in $|\zeta|<\rho$. Let
$H_n=G_{m,n}(\I)$, which is independent of $m$. By (\ref{Gmnest}), $|H_n|\le C$ for all $n\ge 1$. Let $\{H_{n_i}\}$ be any convergent subsequence. If we can show that
\begin{equation}\label{Hnilimit}
\lim_{i\to\infty}H_{n_i}=G(\I),
\end{equation}
we are done since,
\begin{equation*}
I+f_{\I}(K)=I-(I+K)^{-1}K=(I+K)^{-1}.
\end {equation*}
Also,
\begin{equation}\label{detid}
\det(I+|B|)\det(I-|B|)=\det(I-|B|^2)=\prod_{j=1}^\infty(1-\lambda_j^2)=\det(I+K),
\end{equation}
where $\lambda_j$ are the singular values of $B$.
The fact that  $\{G_{m,n}(\zeta)\}$ is a normal family and a diagonal argument shows that there is a subsequence $\{n_{i,i}\}$ of $\{n_i\}$ such that
\begin{equation*}
\lim_{i\to\infty}G_{m,n_{i,i}}(\zeta)=:G_m(\zeta) 
\end{equation*}
uniformly in $|\zeta|\le\rho'<\rho$ for each sufficiently large $m$. By (\ref{Gmnest}), $\{G_m(\zeta)\}$ is a normal family in $|\zeta| <\rho$. Let $\{m_j\}$ be any sequence such that
$G_{m_j}(\zeta)$ converges uniformly in $|\zeta|\le\rho'$, where $1<\rho'<\rho$. Then, for $\zeta$ real, $|\zeta|\le\rho'$,
\begin{equation*}
\lim_{j\to\infty}G_{m_j}(\zeta)=\lim_{j\to\infty}\lim_{i\to\infty}G_{m_j,n_{i,i}}(\zeta)=G(\zeta),
\end{equation*}
by (\ref{Gmnlimit}). Since $\{m_j\}$ was arbitrary, we see that $\lim_{m\to\infty}G_m(\zeta)=G(\zeta)$ uniformly for $\zeta\in\C$, $|\zeta|\le\rho'$. For any fixed m,
\begin{equation*}
\lim_{i\to\infty}H_{n_i}=\lim_{i\to\infty}H_{n_{i,i}}=\lim_{i\to\infty}G_{m,n_{i,i}}(\I)=G_m(\I). 
\end{equation*}
We can now let $m\to\infty$ to get (\ref{Hnilimit}). This completes the proof of
the theorem.

\section{Proof of Lemma \ref{lem:Gmnres}}\label{sect3}
We start by giving a technical lemma that we will use below.
\begin{lemma}\label{lem:tc}
Assume that $\gamma$ is a $C^{5+\alpha}$ curve for some $\alpha>0$ so that the extended exterior mapping function $\phi$ is $C^{5+\alpha}$ on $\T$. Let the operator
$B$ on $\ell^2(\C)$ be defined by \eqref{Bop}. Then $B$ is a trace class operator. Also, if $\delta_m$ is defined by
\begin{equation}\label{deltam}
\delta_m=\left(\sum_{k\vee\ell>m}(k\ell)^{2+\epsilon}|b_{k\ell}|^2\right)^{1/2},
\end{equation}
we have that $\delta_m\to 0$ as $m\to\infty$.
Furthermore there is a constant $C$ so that
\begin{equation}\label{aklsumbound}
\sum_{k\vee\ell>m}(k^2+\ell^2)|a_{k\ell}|\le C
\end{equation}
for all $m\ge 1$.
\end {lemma}

\begin{proof}
Since $\phi$ is $C^{5+\alpha}$ it follows from the definition of the Grunsky coefficients that there is a constant $C$ such that
\begin{equation}\label{aklest1}
|a_{k\ell}|\le \frac C{k^{2+\epsilon}\ell^{2+\epsilon}},
\end{equation}
\begin{equation}\label{aklest2}
|a_{k\ell}|\le \frac C{k^{3+\epsilon}\ell^{1+\epsilon}},
\end{equation}
for some $\epsilon>0$. Consider the operators given by $K=(k^{1+\epsilon/2}\ell^{1+\epsilon/2}a_{k\ell})_{k,\ell\ge 1}$ and $D=(k^{-1/2-\epsilon/2}\delta_{k\ell})_{k,\ell\ge 1}$.
Then $K$ and $D$ are Hilbert-Schmidt operators, and since $B=DKD$ we see that $B$ is a trace class operator.

We see from \eqref{bkl}, \eqref{deltam}, and \eqref{aklest1} that
\begin{equation*}
\delta_m^2=\sum_{k\vee\ell>m}k^{2+\epsilon}\ell^{2+\epsilon}|b_{k\ell}|^2\le\sum_{k\vee\ell>m}\frac{C}{k^{1+\epsilon}\ell^{1+\epsilon}},
\end{equation*}
which $\to 0$ as $m\to\infty$. Also, since $a_{k\ell}$ is symmetric \eqref{aklest2} gives the estimate
\begin{equation*}
\sum_{k\vee\ell>m}(k^2+\ell^2)|a_{k\ell}|=2\sum_{k\vee\ell>m}k^2|a_{k\ell}|\le C
\end{equation*}
for all $m\ge 1$.
\end{proof}

We turn now to the proof of the estimate (\ref{Gmnest}). This proof will also give us an upper bound in (\ref{Gmnlimit}). After that we will prove a lower bound in (\ref{Gmnlimit}) which
will coincide with the upper bound and hence prove the limit. First, we need an estimate of $W_m(\theta)$ defined by (\ref{Wm}). Note that
\begin{equation*}
W_m(\theta)\le\sum_{k\vee\ell>m}|b_{k\ell}||X_k-\I Y_k||X_\ell-\I Y_\ell|=
\lim_{M\to\infty}\sum_{\substack{k\vee\ell>m \\ k\wedge \ell\le M}}|b_{k\ell}||X_k-\I Y_k||X_\ell-\I Y_\ell|.
\end{equation*}
Let $\epsilon>0$ be fixed. By the Cauchy-Schwarz' inequality
\begin{align}\label{Wmest}
&\sum_{\substack{k\vee\ell>m \\ k\wedge \ell\le M}}|b_{k\ell}||X_k-\I Y_k||X_\ell-\I Y_\ell|\\
&\le \left(\sum_{\substack{k\vee\ell>m \\ k\wedge \ell\le M}}(k\ell)^{2+\epsilon}|b_{k\ell}|^2\right)^{1/2}\left(\sum_{\substack{k\vee\ell>m \\ k\wedge \ell\le M}}
\frac 1{k^{2+\epsilon}}|X_k-\I Y_k|^2\frac 1{\ell^{2+\epsilon}}|X_\ell-\I Y_\ell|^2\right)^{1/2}\notag\\
&\le\delta_m\left(\sum_{k=1}^M \frac 1{k^{2+\epsilon}}(X_k^2+Y_k^2)\right).\notag
\end{align}
We know from Lemma \ref{lem:tc} that our assumptions on $\gamma$ imply that $\delta_m\to 0$ as $m\to\infty$. 
Define the $2M$ times $2M$ matrix $D_{m,M}$ by
\begin{equation*}
D_{m,M}^{-1}=\begin{pmatrix} \text{diag\,}(\frac {\delta_m}{k^{2+\epsilon}})_{1\le k\le M}  & 0\\ 0  &  \text{diag\,}(\frac {\delta_m}{k^{2+\epsilon}})_{1\le k\le M}\end{pmatrix}.
\end{equation*}
Because of the inequality (\ref{Gmnineq}) we can assume that $\zeta\in\R$ which we will do from now on.
We see from (\ref{Gmnzeta3}), Fatou's lemma and (\ref{Wmest}) that 
\begin{align}\label{Gmnest2}
G_{m,n}(\zeta)&\le\frac 1{\pi^m}\int_{\R^m}du\int_{\R^m}dv\exp(-\begin{pmatrix} \mathbf{u} \\ \mathbf{v}\end{pmatrix}^t\begin{pmatrix} \mathbf{u} \\ \mathbf{v}\end{pmatrix})\\
&\times\varliminf_{M\to\infty}\E_n\left[\exp\left(2(L_m(\zeta)^t+\mathbf{g}(\zeta))^t\begin{pmatrix} \mathbf{X}\\\mathbf{Y}\end{pmatrix}+
\sum_{\substack{k\vee\ell>m \\ k\wedge \ell\le M}}|b_{k\ell}||X_k-\I Y_k||X_\ell-\I Y_\ell|\right)\right]\notag\\
&\le \frac 1{\pi^m}\int_{\R^m}du\int_{\R^m}dv\exp(-\begin{pmatrix} \mathbf{u} \\ \mathbf{v}\end{pmatrix}^t\begin{pmatrix} \mathbf{u} \\ \mathbf{v}\end{pmatrix})\notag\\
&\times\varliminf_{M\to\infty}\E_n\left[\exp\left(2(L_m(\zeta)^t+\mathbf{g}(\zeta))^t\begin{pmatrix} \mathbf{X}\\\mathbf{Y}\end{pmatrix}+
\begin{pmatrix} P_M\mathbf{X}\\P_M\mathbf{Y}\end{pmatrix}^tD_{m,M}^{-1}\begin{pmatrix} P_M\mathbf{X}\\P_M\mathbf{Y}\end{pmatrix}\right)\right].\notag
\end{align}
We now use the Gaussian integral
\begin{align*}
\exp\left(\begin{pmatrix} P_M\mathbf{X}\\P_M\mathbf{Y}\end{pmatrix}^tD_{m,M}^{-1}\begin{pmatrix} P_M\mathbf{X}\\P_M\mathbf{Y}\end{pmatrix}\right)
&=\frac 1{\pi^M}\left(\prod_{k=1}^M\frac {k^{2+\epsilon}}{\delta_m}\right)\int_{\R^M}dp\int_{\R^M}dq\exp\left(-\begin{pmatrix} \mathbf{p} \\ \mathbf{q}\end{pmatrix}^t D_{m,M}\begin{pmatrix} \mathbf{p} \\ \mathbf{q}\end{pmatrix}\right.\\
&\left.+2\begin{pmatrix} P_M^t\mathbf{p} \\ P_M^t\mathbf{q}\end{pmatrix}^t\begin{pmatrix} \mathbf{X}\\ \mathbf{Y}\end{pmatrix}\right),
\end{align*}
where $\mathbf{p}$ and $\mathbf{q}$ are column vectors in $\R^M$.
If we use this identity in (\ref{Gmnest2}), we obtain the estimate
\begin{align}\label{Gmnest3}
&G_{m,n}(\zeta)\le\frac 1{\pi^{m+M}}\left(\prod_{k=1}^M\frac {k^{2+\epsilon}}{\delta_m}\right)\int_{\R^M}dp\int_{\R^M}dq\int_{\R^m}du\int_{\R^m}dv\\
&\times\varliminf_{M\to\infty}
\exp\left(-\begin{pmatrix} \mathbf{u} \\ \mathbf{v}\end{pmatrix}^t\begin{pmatrix} \mathbf{u} \\ \mathbf{v}\end{pmatrix}
-\begin{pmatrix} \mathbf{p} \\ \mathbf{q}\end{pmatrix}^t D_{m,M}\begin{pmatrix} \mathbf{p} \\ \mathbf{q}\end{pmatrix}\right)
\E_n\left[\exp\left(2\left(L_m(\zeta)+\mathbf{g}(\zeta)+\begin{pmatrix} P_M^t\mathbf{p} \\ P_M^t\mathbf{q}\end{pmatrix}\right)^t\begin{pmatrix} \mathbf{X}\\ \mathbf{Y}\end{pmatrix}\right)\right].\notag
\end{align}
We will now make use of the following upper bound which is a consequence of the strong Szeg\H{o} limit theorem, \cite[p. 268]{Jo1}, or the Geronimo-Case-Borodin-Okounkov identity, \cite[Lemma 2.3]{JoLa}. 
\begin{lemma}\label{lem:ub}
We have the estimate
\begin{equation}\label{ub}
\E_n\left[\exp\left(2\begin{pmatrix} \mathbf{c}\\ \mathbf{d}\end{pmatrix}^t\begin{pmatrix} \mathbf{X}\\ \mathbf{Y}\end{pmatrix}\right)\right]
\le \exp\left(\begin{pmatrix} \mathbf{c}\\ \mathbf{d}\end{pmatrix}^t\begin{pmatrix} \mathbf{c}\\ \mathbf{d}\end{pmatrix}\right),
\end{equation}
for  infinite column vectors $\mathbf{c}=(c_k)_{k\ge 1}$, $\mathbf{d}=(d_k)_{k\ge 1}$ in $\ell^2(\R)$.
\end{lemma}

The estimate (\ref{ub}) gives
\begin{align}\label{Enub}
&\E_n\left[\exp\left(2\left(L_m(\zeta)+\mathbf{g}(\zeta)+\begin{pmatrix} P_M^t\mathbf{p} \\ P_M^t\mathbf{q}\end{pmatrix}\right)^t\begin{pmatrix} \mathbf{X}\\ \mathbf{Y}\end{pmatrix}\right)\right]\\
&\le \exp\left(\left(L_m(\zeta)+\mathbf{g}(\zeta)+\begin{pmatrix} P_M^t\mathbf{p} \\ P_M^t\mathbf{q}\end{pmatrix}\right)^t\left(L_m(\zeta)+\mathbf{g}(\zeta)+\begin{pmatrix} P_M^t\mathbf{p} \\ P_M^t\mathbf{q}\end{pmatrix}\right)\right)\notag\\
&=\exp\left(\left(L_m(\zeta)+\mathbf{g}(\zeta)
 \right)^t\left(L_m(\zeta)+\mathbf{g}(\zeta)\right)
+2 \left(L_m(\zeta)+\mathbf{g}(\zeta) \right)^t\begin{pmatrix} P_M^t\mathbf{p} \\ P_M^t\mathbf{q}\end{pmatrix}
+\begin{pmatrix} \mathbf{p} \\ \mathbf{q}\end{pmatrix}^t\begin{pmatrix} \mathbf{p} \\ \mathbf{q}\end{pmatrix}\right).\notag
\end{align}
Inserting this into (\ref{Gmnest3}), the $pq$-integral becomes
\begin{align}\label{pqint}
&\frac 1{\pi^{M}}\left(\prod_{k=1}^M\frac {k^{2+\epsilon}}{\delta_m}\right)\int_{\R^M}dp\int_{\R^M}dq\exp\left(-\begin{pmatrix} \mathbf{p} \\ \mathbf{q}\end{pmatrix}^t(D_{m,M}-I)
\begin{pmatrix} \mathbf{p} \\ \mathbf{q}\end{pmatrix}\right.\\ 
&\left.+2\left(L_m(\zeta)+\mathbf{g}(\zeta) \right)^t
\begin{pmatrix} P_M & 0 \\ 0 & P_M\end{pmatrix}^t\begin{pmatrix} \mathbf{p} \\ \mathbf{q}\end{pmatrix}\right)\notag\\
&=\left(\prod_{k=1}^M\frac{k^{2+\epsilon}/\delta_m}{k^{2+\epsilon}/\delta_m-1}\right)
\exp\left(\left(L_m(\zeta)+\mathbf{g}(\zeta) \right)^t
\begin{pmatrix} P_M & 0 \\ 0 & P_M\end{pmatrix}^t(D_{m,M}-I)^{-1}\begin{pmatrix} P_M & 0 \\ 0 & P_M\end{pmatrix}\left(L_m(\zeta)+\mathbf{g}(\zeta) \right)\right)\notag\\
&\le\prod_{k=1}^M\frac{1}{1-\delta_m/k^{2+\epsilon}}\exp\left(\frac{\delta_m}{1-\delta_m}\left(L_m(\zeta)+\mathbf{g}(\zeta) \right)^t
\left(L_m(\zeta)+\mathbf{g}(\zeta) \right)\right),\notag
\end{align}
where the last inequality follows from the fact that all entries in $(D_{m,M}-I)^{-1}$ are $\le\delta_m(1-\delta_m)^{-1}$.
If we assume that $m$ is so large that $\delta_m\le 1/2$ then there is a constant $C_\epsilon$ so that
\begin{equation*}
\prod_{k=1}^M\frac{1}{1-\delta_m/k^{2+\epsilon}}\le e^{C_\epsilon\delta_m}.
\end{equation*}
Thus, (\ref{Gmnest3}), (\ref{Enub}) and \eqref{pqint} give
\begin{equation*}
G_{m,n}(\zeta)\le \frac {e^{C_\epsilon\delta_m}}{\pi^m}\int_{\R^m}du\int_{\R^m}dv\exp\left(-\begin{pmatrix} \mathbf{u} \\ \mathbf{v}\end{pmatrix}^t\begin{pmatrix} \mathbf{u} \\ \mathbf{v}\end{pmatrix}+\frac{1}{1-\delta_m}\left(L_m(\zeta)+\mathbf{g}(\zeta) \right)^t\left(L_m(\zeta)+\mathbf{g}(\zeta) \right)\right).
\end{equation*}
We now insert the definition \eqref{Lmzeta} of $L_m(\zeta)$ into the right side. After some computation we get the estimate
\begin{align}\label{Gmnest4}
G_{m,n}(\zeta)&\le\frac {e^{C_\epsilon\delta_m}}{\pi^m}\int_{\R^m}du\int_{\R^m}dv\exp\left(-\begin{pmatrix} \mathbf{u} \\ \mathbf{v}\end{pmatrix}^t
\left(I-\begin{pmatrix}\frac{\zeta^2}{1-\delta_m}\Lambda_m & 0 \\ 0 & \frac 1{1-\delta_m}\Lambda_m\end{pmatrix}\right)\begin{pmatrix} \mathbf{u} \\ \mathbf{v}\end{pmatrix}\right.
\\
&\left.+\frac 2{1-\delta_m}\begin{pmatrix} \mathbf{u} \\ \mathbf{v}\end{pmatrix}^t\begin{pmatrix} \zeta\Lambda_m^{1/2} & 0 \\ 0 & \Lambda_m^{1/2}\end{pmatrix}
T_m^t\begin{pmatrix} P_m & 0\\ 0 & P_m\end{pmatrix}\mathbf{g}(\zeta)+\frac 1{1-\delta_m}\mathbf{g}(\zeta)^t\mathbf{g}(\zeta)\right)\notag
\end{align}
Since $0\le\lambda_{m,k}\le \kappa$, we see that if $|\zeta|\le\rho<1/\sqrt{\kappa}$ and $m$ is so large that $\rho^2\kappa/(1-\delta_m)<1$, then the matrix
$I-\begin{pmatrix}\frac{\zeta^2}{1-\delta_m}\Lambda_m & 0 \\ 0 & \frac 1{1-\delta_m}\Lambda_m\end{pmatrix}$ is positive definite. Hence, we can compute the Gaussian integral in 
(\ref{Gmnest4}) to get the estimate
\begin{align}\label{Gmnest5}
&G_{m,n}(\zeta)\le \frac {e^{C_\epsilon\delta_m}}{\sqrt{\det(I-\frac{\zeta^2}{1-\delta_m}|B_m|)\det(I-\frac{1}{1-\delta_m}|B_m|)}}
\exp\left(\frac 1{1-\delta_m}\begin{pmatrix} P_m\mathbf{a}(\zeta) \\ P_m\mathbf{b}(\zeta)\end{pmatrix}^tT_m\right. \\
&\left.\times\begin{pmatrix} (I-\frac{\zeta^2}{1-\delta_m}\Lambda_m)^{-1}\frac{\zeta^2}{1-\delta_m}\Lambda_m & 0 \\ 0 & 
(I-\frac {1}{1-\delta_m}\Lambda_m)^{-1}\frac{1}{1-\delta_m}\Lambda_m\end{pmatrix}T_m^t\begin{pmatrix} P_m\mathbf{a}(\zeta) \\ P_m\mathbf{b}(\zeta)\end{pmatrix}
+\frac 1{1-\delta_m}\mathbf{g}(\zeta)^t\mathbf{g}(\zeta)\right)\notag
\end{align}
for all $\zeta\in[-\rho,\rho]$ and $m$ sufficiently large. Since $T_m$ is an orthogonal matrix we see that the $\ell^2(\R)\oplus\ell^2(\R)$-norm of 
$T_m^t\begin{pmatrix} P_m\mathbf{a}(\zeta) \\ P_m\mathbf{b}(\zeta)\end{pmatrix}$ is 
\begin{equation*}
\le\sum_{k=1}^ma(\zeta)_k^2+b(\zeta)_k^2\le\rho^2\sum_{k=1}^\infty k(|a_k|^2+|b_k|^2)<\infty,
\end{equation*}
by the assumption (\ref{gcond}). Since $|\zeta|\le \rho$, $\lambda_{m,k}$, and $\rho^2\kappa/(1-\delta_m)<1$, $\delta_m<1/2$ for $m$ sufficiently large, we see that the
expression in the exponent in the right side of (\ref{Gmnest5}) is bounded by a constant. Note that it is proved in Lemma \ref{lem:tc} that $\delta_m\to 0$ as $m\to\infty$ and hence
$C_\epsilon\delta_m\le 1$ if $m$ is sufficiently large. Also, since $B$ is trace class 
\begin{equation*}
\det(I-\frac{\zeta^2}{1-\delta_m}|B_m|)\to\det(I-\zeta^2|B|)
\end{equation*}
as $m\to\infty$ for $|\zeta|\le\rho$. This proves (\ref{Gmnest}) in Lemma  \ref{lem:Gmnres}. If we recall (\ref{fzetalambda}), we se that (\ref{Gmnest5}) gives
\begin{align*}
\varlimsup_{n\to\infty}G_{m,n}(\zeta)&\le\frac {e^{C_\epsilon\delta_m}}{\sqrt{\det(I-\frac{\zeta^2}{1-\delta_m}|B_m|)\det(I-\frac{1}{1-\delta_m}|B_m|)}}\\
&\times\exp\left(\frac 1{1-\delta_m}\begin{pmatrix} P_m\mathbf{a}(\zeta) \\ P_m\mathbf{b}(\zeta)\end{pmatrix}^tf_\zeta\left(\frac 1{1-\delta_m}K_m\right)
\begin{pmatrix} P_m\mathbf{a}(\zeta) \\ P_m\mathbf{b}(\zeta)\end{pmatrix}+\frac 1{1-\delta_m}\mathbf{g}(\zeta)^t\mathbf{g}(\zeta)\right),
\end{align*}
for $\zeta\in\R$, $|\zeta|\le\rho$. We can let $m\to\infty$ in the right side to conclude
\begin{equation}\label{Gmnupper}
\varlimsup_{m\to\infty}\varlimsup_{n\to\infty}G_{m,n}(\zeta)\le\frac {\exp\left(\mathbf{g}(\zeta)^t(I+f_\zeta(K))\mathbf{g}(\zeta)\right)}
{\sqrt{\det(I-\zeta^2|B|)\det(I-|B|)}}=G(\zeta),
\end{equation}
for $\zeta\in\R$, $|\zeta|\le\rho$.

In order to prove (\ref{Gmnlimit}) we also need a lower bound. Fix $D>0$ and let $\zeta\in\R$. We see from (\ref{Gmnzeta3}) that
\begin{align}\label{Gmnlb1}
G_{m,n}(\zeta)\ge\frac 1{\pi^m}\int_{[-D,D]^m}du\int_{[-D,D]^m}dv&\exp(-\begin{pmatrix} \mathbf{u} \\ \mathbf{v}\end{pmatrix}^t\begin{pmatrix} \mathbf{u} \\ \mathbf{v}\end{pmatrix})
\\&\times
\E_n\left[\exp\left(2(L_m(\zeta)+\mathbf{g}(\zeta))^t\begin{pmatrix}\mathbf{X} \\\mathbf{Y} \end{pmatrix}+W_m(\theta)\right)\right].\notag
\end{align}
Let $f(\theta)$ be such that
\begin{equation}\label{ftheta}
2(L_m(\zeta)+\mathbf{g}(\zeta))^t\begin{pmatrix}\mathbf{X} \\\mathbf{Y} \end{pmatrix}=\sum_\mu f(\theta_\mu).
\end{equation}
Note that $f$ is real-valued since $\zeta\in\R$. We want to estimate
\begin{equation*}
\E_n[\exp(\sum_\mu f(\theta_\mu)+W_m(\theta))]
\end{equation*}
from below. To do this we will use an idea from \cite[Lemma 2.3]{Jo1}. Let $h(\theta)$ be a given, smooth $2\pi$-periodic, real-valued function, and let $C(h)$ denote a positive 
constant, whose exact meaning will change, that depends only on $h$ but not on $n$ or $\theta$. Where it occurs below it can be bounded by $||h||_\infty$, $||h'||_\infty$ and
$||h''||_\infty$. Write
\begin{align*}
S_n(\theta)&=\sum_\mu f(\theta_\mu-\frac 1nh(\theta_\mu)),\\
U_n(\theta)&=-\frac 1n\sum_{\mu\neq\nu}\frac 12\cot(\frac{\theta_\mu-\theta_\nu}2)(h(\theta_\mu)-h(\theta_\nu)),\\
V_n(\theta)&=-\frac 1n\sum_\mu h'(\theta_\mu)-\frac 1n\sum_{\mu\neq\nu}\frac {(h(\theta_\mu)-h(\theta_\nu))^2}{\sin^2(\frac{\theta_\mu-\theta_\nu}2)}.
\end{align*}
If we let
\begin{equation*}
\phi_\mu=\theta_\mu-\frac 1nh(\theta_\mu),
\end{equation*}
a Taylor expansion gives
\begin{align}\label{Texp}
&\sum_{\mu\neq\nu}\log|e^{\I\phi_\mu}-e^{\I\phi_\nu}|+\sum_\mu f(\phi_\mu)=
\sum_{\mu\neq\nu}\log\left|2\sin\frac{\theta_\mu-\theta_\nu-\frac 1n (h(\theta_\mu)-h(\theta_\nu))}2\right|+\sum_\mu f(\theta_\mu-\frac 1n h(\theta_\mu))\\
&=\sum_{\mu\neq\nu}\log|e^{\I\theta_\mu}-e^{\I\theta_\nu}|+S_n(\theta)+U_n(\theta)+V_n(\theta)+\frac 1n\sum_\mu h'(\theta_\mu)+R_n^{(1)}(\theta),\notag
\end{align}
where
\begin{equation}\label{Rn1}
|R_n^{(1)}(\theta)|\le\frac{C(h)}n.
\end{equation}
We see from (\ref{Gmnlb1}), (\ref{ftheta}) and (\ref{Texp}) that
\begin{align}\label{Gmnlb2}
G_{m,n}(\zeta)&\ge\frac 1{\pi^m}\int_{[-D,D]^m}du\int_{[-D,D]^m}dv\exp(-\begin{pmatrix} \mathbf{u} \\ \mathbf{v}\end{pmatrix}^t\begin{pmatrix} \mathbf{u} \\ \mathbf{v}\end{pmatrix})
\\&\times
\E_n\left[\exp\left(S_n(\theta)+U_n(\theta)+V_n(\theta)+\sum_{\mu,\nu}w_m(\theta_\mu-\frac 1nh(\theta_\mu),\theta_\nu-\frac 1nh(\theta_\nu))+R_n^{(2)}(\theta)\right)\right],\notag
\end{align}
where
\begin{equation*}
R_n^{(2)}(\theta)=R_n^{(1)}(\theta)+\sum_\mu\log(1-\frac 1nh'(\theta_\mu))+\frac 1nh'(\theta_\mu)
\end{equation*}
by (\ref{Rn1}) satisfies
\begin{equation}\label{Rn2}
|R_n^{(2)}(\theta)|\le\frac{C(h)}n.
\end{equation}
The 1- and 2-point marginal densities for $\E_n[\cdot]$ are given by, \cite{Mec},
\begin{align}\label{p1p2}
p_{1,n}(\theta)&=\frac 1{2\pi}\\
p_{2,n}(\theta_1,\theta_2)&=\frac 1{4\pi^2n(n-1)}\left[n^2-\left|\sum_{j=0}^{n-1}e^{\I j(\theta_1-\theta_2)}\right|^2\right].\notag
\end{align}
These formulas can be used to show that
\begin{equation}\label{UVest}
\E_n[U_n(\theta)]=0,\quad \text{and} \quad \E_n[V_n(\theta)]\ge-\sum_{k=1}^\infty k|h_k|^2,
\end{equation}
where $h_k$ are the complex Fourier coefficients of $h$. Hence, we can use Jensen's inequality in (\ref{Gmnlb2}) to get the estimate
\begin{align}\label{Gmnlb3}
G_{m,n}(\zeta)&\ge\frac 1{\pi^m}\int_{[-D,D]^m}du\int_{[-D,D]^m}dv\exp(-\begin{pmatrix} \mathbf{u} \\ \mathbf{v}\end{pmatrix}^t\begin{pmatrix} \mathbf{u} \\ \mathbf{v}\end{pmatrix})
\\&\times
\exp\left(\E_n[S_n(\theta)]-\sum_{k=1}^\infty k|h_k|^2-\E_n\left[\sum_{\mu,\nu}w_m(\theta_\mu-\frac 1nh(\theta_\mu),\theta_\nu-\frac 1nh(\theta_\nu))\right]-\frac{C(h)}n\right).
\notag
\end{align}
Define
\begin{align}\label{Tn}
T_n^{(1)}&=\E_n[S_n(\theta)]=\frac n{2\pi}\int_{-\pi}^\pi f(\theta-\frac 1n h(\theta))\,d\theta,\\
T_n^{(2)}&=-\frac{n^2}{4\pi^2}\int_{-\pi}^\pi\int_{-\pi}^\pi w_m(\theta_1-\frac 1n h(\theta_1),\theta_2-\frac 1n h(\theta_2))\,d\theta_1d\theta_2,\notag\\
T_n^{(3)}&=\frac 1{4\pi^2}\int_{-\pi}^\pi\int_{-\pi}^\pi w_m(\theta_1-\frac 1n h(\theta_1),\theta_2-\frac 1n h(\theta_2))\left|\sum_{j=0}^{n-1}e^{\I j(\theta_1-\theta_2)}\right|^2
\,d\theta_1d\theta_2\notag\\
 &-\frac n{2\pi}\int_{-\pi}^\pi  w_m(\theta-\frac 1n h(\theta),\theta-\frac 1n h(\theta))\,d\theta.\notag
\end{align}
Then, using (\ref{p1p2}) and (\ref{Gmnlb3}), we find that for $\zeta\in\R$,
\begin{equation}\label{Gmnlb4}
G_{m,n}(\zeta)\ge\frac 1{\pi^m}\int_{[-D,D]^m}du\int_{[-D,D]^m}dv\exp\left(-\begin{pmatrix} \mathbf{u} \\ \mathbf{v}\end{pmatrix}^t\begin{pmatrix} \mathbf{u} \\ \mathbf{v}\end{pmatrix}
+T_n^{(1)}+T_n^{(2)}+T_n^{(3)}-\sum_{k=1}^\infty k|h_k|^2-\frac{C(h)}n\right).
\end{equation}
Note that $f$ depends on $u$ and $v$, and we can choose $h$ to depend on $u$ and $v$ also. Hence $T_n^{(j)}$ depends on $u$ and $v$.
Define, for $n$ sufficiently large (depending on $h$),
\begin{equation}\label{rnsn}
r_n(\theta)=\theta-\frac 1nh(\theta),\quad \text{and} \quad s_n(\theta)=r_n^{-1}(\theta).
\end{equation}
Then, if we write
\begin{equation}\label{snprime}
s_n'(\theta)=1+\frac 1nh'(\theta)+\frac 1{n^2}H_n(\theta),
\end{equation}
we have the bound
\begin{equation}\label{Hn}
|H_n(\theta)|\le C(h).
\end{equation}
By (\ref{Tn}), (\ref{rnsn}), \eqref{snprime}, and the fact that $f$ has zero mean,
\begin{align}\label{Tn1}
T_n^{(1)}&=\frac 1{2\pi}\int_{-\pi}^\pi f(\theta)h'(\theta)\,d\theta+\frac 1{2\pi n}\int_{-\pi}^\pi f(\theta)H_n(\theta)\,d\theta\\
&=-\I\sum_{k\in\mathbb{Z}} kf_{k}h_{-k}+e_n^{(1)},\notag
\end{align}
where
\begin{equation}\label{en1}
|e_n^{(1)}|\le \frac{C(h)}n||f||_1.
\end{equation}
Also,
\begin{align*}
T_n^{(2)}&=-\frac{n^2}{4\pi^2}\int_{-\pi}^\pi\int_{-\pi}^\pi w_m(\theta_1,\theta_2)s_n'(\theta_1)s_n'(\theta_2)\,d\theta_1d\theta_2\\
&=-\re\sum_{k\vee \ell>m}a_{k\ell}\frac{n^2}{4\pi^2}\int_{-\pi}^\pi\int_{-\pi}^\pi e^{-\I k\theta_1-\I\ell\theta_2}s_n'(\theta_1)s_n'(\theta_2)\,d\theta_1d\theta_2\notag.
\end{align*}
By \eqref{snprime} and \eqref{Hn},
\begin{equation*}
\frac{n^2}{4\pi^2}\int_{-\pi}^\pi\int_{-\pi}^\pi e^{-\I k\theta_1-\I\ell\theta_2}s_n'(\theta_1)s_n'(\theta_2)\,d\theta_1d\theta_2=
-k\ell h_kh_\ell+e_n^{(2)}(k,\ell),
\end{equation*}
where
\begin{equation*}
|e_n^{(2)}(k,\ell)|\le \frac{C(h)}n.
\end{equation*}
Thus
\begin{equation}\label{Tn2}
T_n^{(2)}=\re\sum_{k\vee \ell>m}k\ell a_{k\ell}h_kh_\ell-e_n^{(2)},
\end{equation}
where
\begin{equation}\label{en2}
|e_n^{(2)}|=\left|\re\sum_{k\vee \ell>m}a_{k\ell}e_n^{(2)}(k,\ell)\right|\le \frac{C(h)}n,
\end{equation}
by Lemma \ref{lem:tc}.
We now consider $T_n^{(3)}$. By Taylor's theorem
\begin{align*}
w_m(\theta_1-\frac 1n h(\theta_1),\theta_2-\frac 1n h(\theta_2))&=\re\sum_{k\vee \ell>m}a_{k\ell}e^{-\I k\theta_1-\I\ell\theta_2}\\
&-\frac 1n(h(\theta_1)+h(\theta_2))
\re\sum_{k\vee \ell>m}-\I ka_{k\ell}e^{-\I k\theta_1-\I\ell\theta_2}+e_n^{(3)}(\theta_1,\theta_2),
\end{align*}
where
\begin{equation}\label{en3theta}
|e_n^{(3)}(\theta_1,\theta_2)|\le\frac{C(h)}{n^2}\sum_{k\vee \ell>m}(k^2+\ell^2)|a_{k\ell}|\le \frac{C(h)}{n^2},
\end{equation}
by Lemma \ref{lem:tc}.
Thus, by the definition of $T_n^{(3)}$,
\begin{align}\label{Tn3}
T_n^{(3)}&=\re\sum_{k\vee \ell>m}a_{k\ell}\frac{1}{4\pi^2}\int_{-\pi}^\pi\int_{-\pi}^\pi e^{-\I k\theta_1-\I\ell\theta_2}\left|\sum_{j=0}^{n-1}e^{\I j(\theta_1-\theta_2)}\right|^2
\,d\theta_1d\theta_2\\
&-\frac 1n\re\sum_{k\vee \ell>m}-\I ka_{k\ell}\frac{1}{4\pi^2}\int_{-\pi}^\pi\int_{-\pi}^\pi (h(\theta_1)+h(\theta_2))e^{-\I k\theta_1-\I\ell\theta_2}\left|\sum_{j=0}^{n-1}e^{\I j(\theta_1-\theta_2)}\right|^2\,d\theta_1d\theta_2\notag\\
&+2\re\sum_{k\vee \ell>m}a_{k\ell}\frac 1{2\pi}\int_{-\pi}^\pi h(\theta)(-\I ke^{-\I(k+\ell)\theta})\,d\theta+e_N^{(3)}=:I_1+I_2+I_3+e_n^{(3)},\notag
\end{align}
where
\begin{equation*}
e_n^{(3)}=\frac{1}{4\pi^2}\int_{-\pi}^\pi\int_{-\pi}^\pi e_n^{(3)}(\theta_1,\theta_2)\left|\sum_{j=0}^{n-1}e^{\I j(\theta_1-\theta_2)}\right|^2\,d\theta_1d\theta_2
-\frac{n}{2\pi}\int_{-\pi}^\pi e_n^{(3)}(\theta,\theta)\,d\theta.
\end{equation*}
Since
\begin{equation*}
\frac{1}{4\pi^2}\int_{-\pi}^\pi\int_{-\pi}^\pi \left|\sum_{j=0}^{n-1}e^{\I j(\theta_1-\theta_2)}\right|^2\,d\theta_1d\theta_2=n
\end{equation*}
it follows from the estimate \eqref{en3theta} that
\begin{equation}\label{en3}
|e_n^{(3)}|\le \frac{C(h)}n.
\end{equation}
Now,
\begin{align}\label{I1}
I_1&=\re\sum_{k\vee \ell>m}a_{k\ell}\sum_{j_1,j_2=0}^{n-1}\frac{1}{4\pi^2}\int_{-\pi}^\pi\int_{-\pi}^\pi e^{-\I k\theta_1-\I\ell\theta_2+\I(j_1-j_2)(\theta_1-\theta_2)}\,d\theta_1d\theta_2\\
&=\re\sum_{k\vee \ell>m}a_{k\ell}\sum_{j_1,j_2=0}^{n-1}\delta_{k,j_1-j_2}\delta_{\ell,j_2-j_1}=0,
\end{align}
since non-zero Kronecker deltas require $j_1-j_2=k=-\ell$, which is not possible since $k,\ell\ge 1$. Next, we see that
\begin{align*}
I_2&=-\frac 1n\re\sum_{k\vee \ell>m}-\I ka_{k\ell}\sum_{j_1,j_2=0}^{n-1}(h_{k+j_2-j_1}\delta_{\ell,j_2-j_1}+h_{\ell+j_1-j_2}\delta_{k,j_1-j_2})\\
&=-\frac 1n\re\sum_{k\vee \ell>m}-\I ka_{k\ell}(2n-(k+\ell))h_{k+\ell}.
\end{align*}
Finally,
\begin{equation*}
I_3=2\re\sum_{k\vee \ell>m}-\I ka_{k\ell}h_{k+\ell}
\end{equation*}
and thus
\begin{equation*}
I_2+I_3=-\frac 1{2n}\re\sum_{k\vee \ell>m}\I (k+\ell)^2a_{k\ell}h_{k+\ell}.
\end{equation*}
Using Lemma \ref{lem:tc} we see that
\begin{equation*}
|I_2+I_3|\le\frac{C(h)}n,
\end{equation*}
and thus by \eqref{Tn3}, \eqref{en3} and \eqref{I1},
\begin{equation*}
|T_n^{(3)}|\le \frac{C(h)}n.
\end{equation*}
We have shown that
\begin{equation*}
T_n^{(1)}+T_n^{(2)}+T_n^{(3)}\ge -\I \sum_{k\in\mathbb{Z}}kf_kh_{-k}+\re\sum_{k\vee \ell>m}k\ell a_{k\ell}h_kh_\ell-\frac{C(h)}n(1+||f||_1),
\end{equation*}
and inserting this estimate into \eqref{Gmnlb4} gives
\begin{align}\label{Gmnlb5}
G_{m,n}(\zeta)&\ge\frac 1{\pi^m}\int_{[-D,D]^m}du\int_{[-D,D]^m}dv\exp\left(-\begin{pmatrix} \mathbf{u} \\ \mathbf{v}\end{pmatrix}^t\begin{pmatrix} \mathbf{u} \\ \mathbf{v}\end{pmatrix}
-\sum_{k=1}^\infty k h_kh_{-k}-\I \sum_{k\in\mathbb{Z}}kf_kh_{-k}\right.\\
&\left.+\re\sum_{k\vee \ell>m}k\ell a_{k\ell}h_kh_\ell-\frac{C(h)}n(1+||f||_1)\right).\notag
\end{align}

We now choose $h_k=-\I\sgn(k)f_k$, $1\le |k|\le m$, $h_k=0$ if $|k|>m$ or $k=0$, so that $h$ is a cut-off of the Fourier series for the conjugate function to $f$. Then
\begin{equation*}
-\I \sum_{k\in\mathbb{Z}}kf_kh_{-k}=2\sum_{k=1}^mk|f_k|^2, \quad \text{and} \quad \sum_{k=1}^\infty k h_kh_{-k}=\sum_{k=1}^mk|f_k|^2,
\end{equation*}
and
\begin{equation*}
\re\sum_{k\vee \ell>m}k\ell a_{k\ell}h_kh_\ell=0.
\end{equation*}
Hence, from \eqref{Gmnlb5} we see that for $\zeta\in\R$,
\begin{equation*}
G_{m,n}(\zeta)\ge\frac 1{\pi^m}\int_{[-D,D]^m}du\int_{[-D,D]^m}dv\exp\left(-\begin{pmatrix} \mathbf{u} \\ \mathbf{v}\end{pmatrix}^t\begin{pmatrix} \mathbf{u} \\ \mathbf{v}
\end{pmatrix}+\sum_{k=1}^mk|f_k|^2-\frac{C(h)}n(1+||f||_1)\right).
\end{equation*}
With $m$ fixed and for $u,v\in [-D,D]^m$ with $D$ fixed, and $|\zeta|\le\rho$, we see that $C(h)$ is bounded and thus
\begin{equation}\label{Gmnlb6}
\varliminf_{n\to\infty}G_{m,n}(\zeta)\ge\frac 1{\pi^m}\int_{[-D,D]^m}du\int_{[-D,D]^m}dv\exp\left(-\begin{pmatrix} \mathbf{u} \\ \mathbf{v}\end{pmatrix}^t\begin{pmatrix} \mathbf{u} \\ \mathbf{v}
\end{pmatrix}+\sum_{k=1}^mk|f_k|^2\right).
\end{equation}
We see from \eqref{ftheta} that
\begin{equation*}
\sum_{k=1}^mk|f_k|^2=(L_m(\zeta)+P_m\mathbf{g}(\zeta))^t(L_m(\zeta)+P_m\mathbf{g}(\zeta)).
\end{equation*}
In \eqref{Gmnlb6} we can let $D\to\infty$ so that the integration in the right side is over $\R^m$ and compute the Gaussian integral. The same computations that led to 
\eqref{Gmnest5} then give
\begin{align*}
\varliminf_{n\to\infty}G_{m,n}(\zeta)&\ge \frac 1{\sqrt{\det(I-\frac{\zeta^2}{1-\delta_m}|B_m|)\det(I-\frac{1}{1-\delta_m}|B_m|)}}
\exp\left(\begin{pmatrix} P_m\mathbf{a}(\zeta) \\ P_m\mathbf{b}(\zeta)\end{pmatrix}^tT_m\right. \\
&\left.\times\begin{pmatrix} (I-\zeta^2\Lambda_m)^{-1}\zeta^2\Lambda_m & 0 \\ 0 & 
(I-\Lambda_m)^{-1}\Lambda_m\end{pmatrix}T_m^t\begin{pmatrix} P_m\mathbf{a}(\zeta) \\ P_m\mathbf{b}(\zeta)\end{pmatrix}
+\begin{pmatrix} P_m\mathbf{a}(\zeta) \\ P_m\mathbf{b}(\zeta)\end{pmatrix}^t\begin{pmatrix} P_m\mathbf{a}(\zeta) \\ P_m\mathbf{b}(\zeta)\end{pmatrix}\right)\notag.
\end{align*}
We can now let $m\to\infty$, and the same computations as previously then give
\begin{equation*}
\varliminf_{m\to\infty}\varliminf_{n\to\infty}G_{m,n}(\zeta)\ge G(\zeta),
\end{equation*}
for $\zeta\in[-\rho,\rho]$ which is what we wanted to prove.

\section{Proof of Theorem \ref{thm:WP}}\label{sec:WPproof}
Without loss of generality we can assume that $\text{\rm cap}(\gamma)=1$ and we will do so in this section.
Consider the function $E_n(r)$ defined by \eqref{Enr}. We have the following lemma.

\begin{lemma}\label{lem:Enrincr}
The sequence of functions $E_n(r)$, $n\ge 1$ is increasing.
\end{lemma}
The lemma will be proved below. We will now prove Theorem \ref{thm:WP} using Lemma \ref{lem:Enrdecr} and Lemma \ref{lem:Enrincr}.

\begin{proof}[Proof of Theorem \ref{thm:WP}]
Let $B_r$, $r>1$ be the Grunsky operator for the curve $\gamma_r$. Then for each $r>1$, by \eqref{logdetlimit},
\begin{equation}\label{Er}
E(r):=\lim_{n\to\infty}E_n(r)=-\frac 12\log\det(I-B_r^*B_r)
\end{equation}
since $\gamma_r$ satisfies the conditions of Theorem \ref{thm:main}. Assume first that $\gamma$ is a Weil-Petersson quasicircle. Then the Grunsky operator $B$ for
$\gamma$ is a Hilbert-Schmidt operator so $B^*B$ is a trace-class operator and consequently it follows from \eqref{Er} that
\begin{equation}\label{Erlimit}
E:=\lim_{r\to 1+}E(r)=-\frac 12\log\det(I-B^*B)<\infty,
\end{equation}
which can be seen by expressing the Grunsky coefficients for $B_r$ in terms of the Grunsky coefficients of $B$, see below.
Since $E_n(r)$ is increasing in $n$, we have that $E_n(r)\le E(r)$ and combining this with \eqref{Erlimit} we obtain
\begin{equation*}
E_n:=\log\frac{Z_n(\gamma)}{(2\pi)^n}=\lim_{r\to 1+}E_n(r)\le\lim_{r\to 1+}E(r)=E<\infty
\end{equation*}
for all $n\ge 1$. Hence
\begin{equation}\label{Znupper}
\limsup_{n\to\infty}\frac{Z_n(\gamma)}{(2\pi)^n}\le E<\infty,
\end{equation}
which proves \eqref{ZnWPbound}. It remains to prove that we also get the right limit. From the monotonicity in $r$ we have that $E_n(r)\le E_n(r')$ if $1<r<r'$, and letting
$r'\to 1+$ gives $E_n(r)\le E_n$ for all $r>1$, $n\ge 1$. Taking the limit $n\to\infty$ gives $E(r)\le\varliminf_{n\to\infty} E_n$ for all $r>1$. Finally, we can let $r\to 1+$ to obtain
$E\le\varliminf_{n\to\infty} E_n$, which combined with \eqref{Znupper} gives,
\begin{equation*}
\lim_{n\to\infty}\log\frac{Z_n(\gamma)}{(2\pi)^n}=E=-\frac 12\log\det(I-B^*B),
\end{equation*}
which is what we wanted to prove.

Next we want to prove that if \eqref{ZnWPbound} holds, then $\gamma$ is a Weil-Petersson quasicircle. It follows from Lemma \ref{lem:Enrdecr} and the definition
\eqref{Zndef} that
\begin{equation}\label{Znineq}
\frac{Z_n(\gamma_r)}{(2\pi)^n}\le \frac{Z_n(\gamma)}{(2\pi)^n}
\end{equation}
for any $r>1$. We can use \eqref{logdetlimit} and take the limit $n\to\infty$ in \eqref{Znineq} to obtain
\begin{equation}\label{detbound}
(\det(I-B_r^*B_r))^{-1/2}\le \varlimsup_{n\to\infty}\frac{Z_n(\gamma)}{(2\pi)^n}=:A<\infty
\end{equation}
for any $r>1$. From \eqref{logphiq} we see that if the Grunsky coefficients for $\gamma$ are $b_{k\ell}$, ${k,\ell\ge 1}$, then the Grunsky coefficients for $\gamma_r$ are
$b_{k\ell}/r^{k+\ell}$, ${k,\ell\ge 1}$. Let $\lambda_j(r)$ be the singular values of $B_r$. Then \eqref{detbound} gives the inequality
\begin{equation*}
\prod_{j=1}^\infty (1-\lambda_j(r)^2)\ge A^{-2},
\end{equation*}
so
\begin{equation*}
||B_r||^2_{HS}=\sum_{j=1}^\infty \lambda_j(r)^2\le -\sum_{j=1}^\infty\log(1-\lambda_j(r)^2)\le 2\log A
\end{equation*}
for all $r>1$. Thus
\begin{equation*}
\sum_{k,\ell=1}^\infty\frac{|b_{k\ell}|^2}{2r^{k+\ell}}\le 2\log A,
\end{equation*}
and letting $r\to 1+$ shows that $B$ is a Hilbert-Schmidt operator, so $\gamma$ is a Weil-Petersson quasicircle.
\end{proof}

Lemma \ref{lem:Enrdecr} will follow from the following lemma.

\begin{lemma}\label{lem:Enrconv}
The function $E_n(r)$ defined by \eqref{Enr} satisfies
\begin{equation}\label{Enrconv}
rE_n''(r)+E_n'(r)\ge 0
\end{equation}
for all $r>1$. Furthermore
\begin{equation}\label{Enrlimit}
\lim_{r\to\infty}E_n(r)=0.
\end{equation}
\end{lemma}
\begin{proof}
Note that by definition
\begin{equation}\label{Zngammarformula}
Z_n(\gamma_r)=\frac 1{n!r^{n(n-1)}}\int_{[-\pi,\pi]^n}\exp\bigg(\re\bigg(\sum_{\mu\neq\nu}\log(\phi(re^{\I\theta_{\mu}})-\phi(re^{\I\theta_{\nu}}))+\sum_\mu\log\phi'(re^{\I\theta_\mu})
\bigg)\bigg)\,d\theta,
\end{equation}
where we used $\phi_r(z)=\phi(rz)/r$ and $\phi_r'(z)=\phi'(rz)$. Making the change of variables $\theta_\mu\mapsto\theta_\mu+\alpha$ for some real $\alpha$ in the right
side of \eqref{Zngammarformula} does not change its value so
\begin{equation}\label{Zngammarformula2}
Z_n(\gamma_r)=\frac 1{n!r^{n(n-1)}}\int_{[-\pi,\pi]^n}e^{F(r,\alpha,\theta)}\,d\theta,
\end{equation}
where
\begin{equation}\label{Fralpha}
F(r,\alpha,\theta)=\sum_{\mu\neq\nu}\log(\phi(re^{\I(\theta_\mu+\alpha)})-\phi(re^{\I(\theta_\nu+\alpha)}))+\sum_\mu\log\phi'(re^{\I(\theta_\mu+\alpha})).
\end{equation}
From the definition of $E_n(r)$ we then obtain
\begin{equation}\label{Enrformula}
E_n(r)=-\log((2\pi)^nn!)-n(n-1)\log r+\log\int_{[-\pi,\pi]^n}e^{F(r,\alpha,\theta)}\,d\theta.
\end{equation}
Using this formula we can compute the derivatives of $E_n(r)$ which gives
\begin{equation}\label{Enrprime}
E_n'(r)=-\frac{n(n-1)}r+\frac{\int(\re\partial_rF)e^{\re F}\,d\theta}{\int e^{\re F}\,d\theta},
\end{equation}
and
\begin{equation}\label{Enrbis}
E_n''(r)=\frac{n(n-1)}r+\frac{\int(\re\partial_r^2F+(\re\partial_rF)^2)e^{\re F}\,d\theta}{\int e^{\re F}\,d\theta}
-\left(\frac{\int(\re\partial_rF)e^{\re F}\,d\theta}{\int e^{\re F}\,d\theta}\right)^2,
\end{equation}
where the integrals are over $[-\pi,\pi]^n$. If we take the derivative with respect to $\alpha$ in \eqref{Enrformula} we get similarly
\begin{equation}\label{Enrprimealpha}
\frac{\int(\re\partial_\alpha F)e^{\re F}\,d\theta}{\int e^{\re F}\,d\theta}=0,
\end{equation}
and
\begin{equation}\label{Enrbisalpha}
\frac{\int(\re\partial_r^2F+(\re\partial_\alpha F)^2)e^{\re F}\,d\theta}{\int e^{\re F}\,d\theta}
-\left(\frac{\int(\re\partial_\alpha F)e^{\re F}\,d\theta}{\int e^{\re F}\,d\theta}\right)^2=0.
\end{equation}
From the definition \eqref{Fralpha} we see that
\begin{equation*}
\partial_rF=\sum_{\mu\neq\nu}\frac{e^{\I(\theta_\mu+\alpha)}\phi'(re^{\I(\theta_\mu+\alpha)})-e^{\I(\theta_\nu+\alpha)}\phi'(re^{\I(\theta_\nu+\alpha)})}
{\phi(e^{\I(\theta_\mu+\alpha)})-\phi(e^{\I(\theta_\nu+\alpha)})}+\sum_\mu\frac{e^{\I(\theta_\mu+\alpha)}\phi''(re^{\I(\theta_\mu+\alpha)})}{\phi'(e^{\I(\theta_\mu+\alpha)})},
\end{equation*}
and
\begin{align*}
\partial_r^2F&=\sum_{\mu\neq\nu}\left[\frac{(e^{\I(\theta_\mu+\alpha)})^2\phi''(re^{\I(\theta_\mu+\alpha)})-(e^{\I(\theta_\nu+\alpha)})^2\phi''(re^{\I(\theta_\nu+\alpha)})}
{\phi(e^{\I(\theta_\mu+\alpha)})-\phi(e^{\I(\theta_\nu+\alpha)})}\right.\\
&-\left.\left(\frac{e^{\I(\theta_\mu+\alpha)}\phi'(re^{\I(\theta_\mu+\alpha)})-e^{\I(\theta_\nu+\alpha)}\phi'(re^{\I(\theta_\nu+\alpha)})}
{\phi(e^{\I(\theta_\mu+\alpha)})-\phi(e^{\I(\theta_\nu+\alpha)})}\right)^2\right]\\
&+\sum_\mu\left[\frac{(e^{\I(\theta_\mu+\alpha)})^2\phi'''(re^{\I(\theta_\mu+\alpha)})}{\phi'(e^{\I(\theta_\mu+\alpha)})}
-\left(\frac{e^{\I(\theta_\mu+\alpha)}\phi''(re^{\I(\theta_\mu+\alpha)})}{\phi'(e^{\I(\theta_\mu+\alpha)})}\right)^2\right].
\end{align*}
An analogous computation gives
\begin{equation*}
\partial_\alpha F=\I r\partial_rF,\quad \text{and}\quad \partial_\alpha^2F=-r\partial_rF-r^2\partial_r^2F.
\end{equation*}
Consequently,
\begin{equation*}
\re\partial_\alpha F=-r\im\partial_rF, \quad \text{and}\quad
\re\partial_\alpha^2F=-r\re\partial_r F-r^2\re\partial_r^2F.
\end{equation*}
If we insert these relations into \eqref{Enrbisalpha}, we get
\begin{equation*}
\frac{\int(-r\re\partial_rF-r^2\re\partial_r^2F+r^2(\im\partial_rF)^2)e^{\re F}\,d\theta}{\int e^{\re F}\,d\theta}-
\left(\frac{\int(r\im\partial_rF)e^{\re F}\,d\theta}{\int e^{\re F}\,d\theta}\right)^2=0,
\end{equation*}
which gives
\begin{align*}
&\frac{r^2\int(\re\partial_r^2F)e^{\re F}\,d\theta}{\int e^{\re F}\,d\theta}=-\frac{r\int(\re\partial_rF)e^{\re F}\,d\theta}{\int e^{\re F}\,d\theta}+
\frac{r^2\int(\im\partial_rF)^2e^{\re F}\,d\theta}{\int e^{\re F}\,d\theta}-r^2\left(\frac{\int(\im\partial_rF)e^{\re F}\,d\theta}{\int e^{\re F}\,d\theta}\right)^2\\
&=-rE_n'(r)-n(n-1)+\frac{r^2\int(\im\partial_rF)^2e^{\re F}\,d\theta}{\int e^{\re F}\,d\theta}-r^2\left(\frac{\int(\im\partial_rF)e^{\re F}\,d\theta}{\int e^{\re F}\,d\theta}\right)^2,
\end{align*}
where the last equality follows from \eqref{Enrprime}. We can use this identity in \eqref{Enrbis} to find
\begin{align*}
r^2E_n''(r)&=n(n-1)-rE_n'(r)-n(n-1)+
r^2\left[\frac{\int(\im\partial_rF)^2e^{\re F}\,d\theta}{\int e^{\re F}\,d\theta}-\left(\frac{\int(\im\partial_rF)e^{\re F}\,d\theta}{\int e^{\re F}\,d\theta}\right)^2\right.\\
&+\left.\frac{\int(\re\partial_rF)^2e^{\re F}\,d\theta}{\int e^{\re F}\,d\theta}-\left(\frac{\int(\re\partial_rF)e^{\re F}\,d\theta}{\int e^{\re F}\,d\theta}\right)^2\right].
\end{align*}
This can be written 
\begin{align*}
&\frac 1{r^2}(r^2E_n''(r)+rE_n'(r))(\int e^{\re F}\,d\theta)^2
=\frac 12\int d\theta\int d\theta'\big[(\im\partial_rF(r,\alpha,\theta)-\im\partial_rF(r,\alpha,\theta'))^2\\
&+(\re\partial_rF(r,\alpha,\theta)-\re\partial_rF(r,\alpha,\theta'))^2\big]e^{\re F(r,\alpha,\theta)+\re F(r,\alpha,\theta')}\ge 0,
\end{align*}
and we have proved the inequality \eqref{Enrconv}.

We have that
\begin{equation}\label{Znphi}
\frac{Z_n(\gamma_r)}{(2\pi)^n}=\frac 1{n!}\int_{[-\pi,\pi]^n}\prod_{\mu\neq\nu}|\phi_r(e^{\I\theta_\mu})-\phi_r(e^{\I\theta_\nu})|\prod_\mu|\phi_r'(e^{\I\theta_\nu})|\,d\theta.
\end{equation}
Since the series \eqref{phiexp} is absolutely convergent for $|z|>1$, there is a constant $C$ so that $|\phi_{-k}|\le C2^k$ for all $k\ge 1$. Now, by  \eqref{phiexp}
\begin{equation}\label{phirphirprime}
\phi_r(z)=z+\sum_{k=0}^\infty\frac{\phi_{-k}}{r^{k+1}z^k}, \quad \text{and} \quad \phi_r'(z)=1+\sum_{k=1}^\infty\frac{k\phi_{-k}}{r^{k+1}z^{k+1}},
\end{equation}
and consequently $\phi_r(z)\to z$ and $\phi_r'(z)\to 1$ uniformly for $z\in\T$ as $r\to\infty$. Hence, we can take the limit $r\to\infty$ in \eqref{Znphi} to obtain
\begin{equation*}
\lim_{r\to\infty}
\frac{Z_n(\gamma_r)}{(2\pi)^n}=\frac 1{(2\pi)^nn!}\int_{[-\pi,\pi]^n}\prod_{\mu\neq\nu}|e^{\I\theta_\mu}-e^{\I\theta_\nu}|\,d\theta=1.
\end{equation*}
This proves \eqref{Enrlimit} and we are done.
\end{proof}

Now we can give the
\begin{proof}[Proof of Lemma \ref{lem:Enrdecr}]
Assume that $E_n'(r_0)>0$ for some $r_0>1$. From \eqref{Enrconv} we see that $rE_n'(r)$ is increasing and hence $rE_n'(r)\ge r_0E_n'(r_0)$ for $r\ge r_0$. Thus,
\begin{equation*}
E_n(r)\ge E_n(r_0)+r_0E_n'(r_0)\int_{r_0}^r\frac{ds}s=E_n(r_0)+r_0E_n'(r_0)\log(r/r_0).
\end{equation*}
If we let $r\to\infty$ this contradicts \eqref{Enrlimit}. Consequently, $E_n'(r)\le 0$ for all $r>1$.
\end{proof}

We turn now to the proof of Lemma \ref{lem:Enrincr}.

\begin{proof}[Proof of Lemma \ref{lem:Enrincr}]
Let $\Pi_n$ be the set of all polynomials of degree $\le n$ with leading coefficient $=1$. Then, see \cite[Sec. 16.2]{Sz}, \cite{Sz2}, we have that
\begin{equation}\label{minpol}
\frac{Z_{n+1}(\gamma_r)/(2\pi)^{n+1}}{Z_{n}(\gamma_r)/(2\pi)^{n}}=\frac{D_{n+1}(1)/(2\pi)^{n+1}}{D_{n}(1)/(2\pi)^{n}}=\frac 1{\kappa_n^2}
=\min_{p\in\Pi_n}\frac 1{2\pi}\int_{\gamma_r}|p(\zeta)|^2\,|d\zeta|,
\end{equation}
and the minimum is attained if and only if $p(\zeta)=\pi_n(\zeta):=\frac 1{\kappa_n}p_n(\zeta)=\zeta^n+\dots$, where $p_n$ are the orthonormal polynomials
with respect to $\gamma_r$,
\begin{equation*}
\int_{\gamma_r}p_m(\zeta)\overline{p_n(\zeta)}\,|d\zeta|=\delta_{mn}.
\end{equation*}
Hence, 
\begin{equation}\label{Enrdifference}
E_{n+1}(r)-E_n(r)=\log\bigg(\frac 1{2\pi}\int_{\gamma_r}|\pi_n(\zeta)|^2\,|d\zeta|\bigg).
\end{equation}
Note that $\phi_r$ is analytic in $|z|>\rho^{-1}$ if $1\le\rho<r$. Fix $\rho\in(1,\infty)$. Then, by \eqref{phirphirprime},
\begin{equation}\label{phiexp2}
z\phi_r(\frac 1z)=1+\sum_{k=0}^\infty\frac{\phi_{-k}}{r^{k+1}}z^{k+1},\quad \text{and} \quad h_r(z):=\phi_r´(\frac 1z)=1+\sum_{k=0}^\infty\frac{k\phi_{-k}}{r^{k+1}}z^{k+1},
\end{equation}
and these functions are analytic in $|z|<\rho$. 
By \eqref{Enrdifference} and Jensen's inequality,
\begin{align}\label{Enrdiffest}
&E_{n+1}(r)-E_n(r)=\log\bigg(\frac 1{2\pi}\int_{-\pi}^\pi|\pi_n(\phi_r(e^{\I\theta}))|^2|\phi_r'(e^{\I\theta})|\,d\theta\bigg)\\
&\ge\frac 1{2\pi}\int_{-\pi}^\pi 2\log|\pi_n(\phi_r(e^{\I\theta}))|+\log|\phi_r'(e^{\I\theta})|\,d\theta
=\frac 1{2\pi}\int_{-\pi}^\pi 2\log|\pi_n(\phi_r(e^{-\I\theta}))|+\log|\phi_r'(e^{-\I\theta})|\,d\theta.\notag
\end{align}
Note that
\begin{equation}\label{pinid}
|\pi_n(\phi_r(e^{-\I\theta}))|=|(e^{\I\theta})^n\pi_n(\phi_r(e^{-\I\theta}))|=|\psi_{n,r}(e^{\I\theta})|,
\end{equation}
where
\begin{equation*}
\psi_{n,r}(z)=z^n\pi_n(\phi_r(\frac 1z)).
\end{equation*}
If $\pi_n(z)=\sum_{j=0}^na_jz^j$, with $a_n=1$, then
\begin{equation}\label{psinrexp}
\psi_{n,r}(z)=z^n\sum_{j=0}^na_j\phi_r(\frac 1z)^j=\sum_{j=0}^na_jz^{n-j}\big(z\phi_r(\frac 1z)\big)^j,
\end{equation}
so we see that $\psi_{n,r}$ is analytic in $|z|<\rho$. Hence, $\log|\psi_{n,r}(z)|$ and $\log|h_r(z)|$ are subharmonic functions in  $|z|<\rho$, and we see from 
\eqref{Enrdiffest} and \eqref{pinid} that
\begin{equation}\label{Enrdiffest2}
E_{n+1}(r)-E_n(r)\ge\frac 1{2\pi}\int_{-\pi}^\pi 2\log|\psi_{n,r}(e^{\I\theta})|+\log|h_r(e^{\I\theta})|\,d\theta
\ge 2\log\psi_{n,r}(0)+\log|h_r(0)|.
\end{equation}
It follows from \eqref{phiexp2}  that $h_r(0)=1$, and from \eqref{phiexp2} and \eqref{psinrexp},
\begin{equation*}
\psi_{n,r}(0)=\sum_{j=0}^na_j 0^{n-j}\cdot 1^j=a_n=1.
\end{equation*}
Consequently, \eqref{Enrdiffest2} gives $E_{n+1}(r)-E_n(r)\ge 0$.
\end{proof}

\section{Heuristic argument for the $\beta$-ensemble}\label{sec:beta}
We will use the same notation, sometimes slightly modified, as in the previous sections and only sketch the argument. Let
\begin{equation*}
\E_{n,\beta}[(\cdot)]=\frac 1{Z_{n,\beta}(\T)n!}\int_{[-\pi,\pi]^n}\prod_{\mu\neq\nu}|e^{\I\theta_\mu}-e^{\I\theta_\nu}|^{\beta/2}(\cdot)\,d\theta
\end{equation*}
denote expectation with respect to the $\beta$-ensemble on the unit circle. As in \eqref{Dn1}, we see that
\begin{align}\label{Dnbeta}
D_{n,\beta}[e^g]&=Z_{n,\beta}(\T)\text{\rm cap}(\gamma)^{\frac{\beta n^2}2}\E_{n,\beta}\left[
\exp\left(-\frac{\beta}2\re\sum_{k,\ell=1}^\infty a_{k\ell}\left(\sum_\mu e^{-\I k\theta_\mu}\right)\left(\sum_\nu e^{-\I\ell\theta_\nu}\right)\right.\right.\\
&\left.\left.+(1-\frac{\beta}2)\sum_\mu\log|\phi'(e^{\I\theta_\mu})|+\sum_\mu g(\phi(e^{\I\theta_\mu}))\right)\right]\notag\\
&=Z_{n,\beta}(\T)\text{\rm cap}(\gamma)^{\frac{\beta n^2}2}\E_{n,\beta}\left[\exp\left(-\frac{\beta}2\re\sum_{k,\ell=1}^\infty a_{k\ell}\left(\sum_\mu e^{-\I k\theta_\mu}\right)\left(\sum_\nu e^{-\I\ell\theta_\nu}\right)+2\mathbf{g}_\beta^t\begin{pmatrix} \mathbf{X} \\ \mathbf{Y}\end{pmatrix}\right)\right],\notag
\end{align}
where we used \eqref{logphiprime}. If we write
\begin{equation*}
\sum_{k,\ell=1}^\infty a_{k\ell}\left(\sum_\mu e^{-\I k\theta_\mu}\right)\left(\sum_\nu e^{-\I\ell\theta_\nu}\right)=
\lim_{m\to\infty}\sum_{k,\ell=1}^m a_{k\ell}\left(\sum_\mu e^{-\I k\theta_\mu}\right)\left(\sum_\nu e^{-\I\ell\theta_\nu}\right)
\end{equation*}
in \eqref{Dnbeta}, take the limit outside the expectation and then interchange the order of the $m\to\infty$ and $n\to\infty$ limits, we are led to study the
limit
\begin{equation*}
\lim_{m\to\infty}\lim_{n\to\infty}\E_{n,\beta}\left[\exp\left(-\frac{\beta}2\re\sum_{k,\ell=1}^m b_{k\ell}(X_k-\I Y_k)(X_\ell-\I Y_\ell)+2\mathbf{g}_\beta^t\begin{pmatrix} \mathbf{X} \\ \mathbf{Y}\end{pmatrix}\right)\right].
\end{equation*}
It seems that it is not easy to justify changing the order of the limits but doing so leads, as we will see, to the conjecture \eqref{betalimit}. 
Set $M_{m,\beta}=\sqrt{\frac{\beta}2}M_m(\I)$, with $M_m(\I)$ as in \eqref{Mmzeta}. We can then use a Gaussian integral to write
\begin{align}\label{Enbeta1}
&\E_{n,\beta}\left[\exp\left(-\frac{\beta}2\re\sum_{k,\ell=1}^m b_{k\ell}(X_k-\I Y_k)(X_\ell-\I Y_\ell)+2\mathbf{g}_\beta^t\begin{pmatrix} \mathbf{X} \\ \mathbf{Y}\end{pmatrix}\right)\right]\\
&=\frac 1{\pi^m}\int_{\R^m}du\int_{\R^m}dv\exp(-\begin{pmatrix} \mathbf{u} \\ \mathbf{v} \end{pmatrix}^t\begin{pmatrix} \mathbf{u} \\ \mathbf{v} \end{pmatrix})
\E_{n,\beta}\left[\exp\left(2\begin{pmatrix} \mathbf{u} \\ \mathbf{v} \end{pmatrix}^tM_{m,\beta}+2\mathbf{g}_\beta^t\begin{pmatrix} \mathbf{X} \\ \mathbf{Y}\end{pmatrix}
\right)\right]\notag\\
&=\frac 1{\pi^m}\int_{\R^m}du\int_{\R^m}dv\exp(-\begin{pmatrix} \mathbf{u} \\ \mathbf{v} \end{pmatrix}^t\begin{pmatrix} \mathbf{u} \\ \mathbf{v} \end{pmatrix})
\E_{n,\beta}\left[\exp\left(2(L_{m,\beta}+\mathbf{g}_\beta)^t\begin{pmatrix} \mathbf{X} \\ \mathbf{Y}\end{pmatrix}\right)\right],\notag
\end{align}
where
\begin{equation*}
L_{m,\beta}=\sqrt{\frac{\beta}2}L_m(\I)=\sqrt{\frac{\beta}2}\begin{pmatrix} P_m & 0 \\  0  &  P_m \end{pmatrix}T_m\begin{pmatrix} \I \Lambda_m^{1/2} & 0 \\ 0  &  \Lambda_m^{1/2}\end{pmatrix}
\begin{pmatrix} \mathbf{u} \\ \mathbf{v} \end{pmatrix}.
\end{equation*}
We can now use the strong Szeg\H{o} limit theorem for the $\beta$-ensemble on the unit circle to take the $n\to\infty$ limit in the last expectation in \eqref{Enbeta1}. This gives the limit
\begin{equation}\label{Enbeta2}
\frac 1{\pi^m}\int_{\R^m}du\int_{\R^m}dv\exp(-\begin{pmatrix} \mathbf{u} \\ \mathbf{v} \end{pmatrix}^t\begin{pmatrix} \mathbf{u} \\ \mathbf{v} \end{pmatrix})
\exp\left(\frac{2}{\beta}(L_{m,\beta}+\mathbf{g}_\beta)^t (L_{m,\beta}+\mathbf{g}_\beta)\right).
\end{equation}
Now,
\begin{equation*}
\frac 2{\beta}L_{m,\beta}^tL_{m,\beta}=-\begin{pmatrix} \mathbf{u} \\ \mathbf{v} \end{pmatrix}^t\left(I-\begin{pmatrix} -\Lambda_m & 0 \\ 0 & \Lambda_m\end{pmatrix}\right)
\begin{pmatrix} \mathbf{u} \\ \mathbf{v} \end{pmatrix},
\end{equation*}
and
\begin{equation*}
\frac 4{\beta}L_{m,\beta}\mathbf{g}_\beta=2\sqrt{\frac 2{\beta}}\begin{pmatrix} \mathbf{u} \\ \mathbf{v} \end{pmatrix}^t
\begin{pmatrix} \I\Lambda_m^{1/2} & 0 \\ 0 & \Lambda_m^{1/2}\end{pmatrix}T_m\begin{pmatrix} P_m & 0 \\  0  &  P_m \end{pmatrix}\mathbf{g}_\beta.
\end{equation*}
We can now perform the Gaussian integrations in \eqref{Enbeta2} to get
\begin{equation*}
\frac 1{\sqrt{\det(I+K_m)}}\exp\left(-\frac 2{\beta}\mathbf{g}_\beta\begin{pmatrix} P_m & 0 \\  0  &  P_m \end{pmatrix}(I+K_m)^{-1}K_m\begin{pmatrix} P_m & 0 \\  0  &  P_m \end{pmatrix}\mathbf{g}_\beta
+\frac 2{\beta}\mathbf{g}_\beta^t\mathbf{g}_\beta\right).
\end{equation*}
If we take the $m\to\infty$ limit of this expression we obtain the right side of \eqref{betalimit}.



\begin{thebibliography}{99}

\itemsep=\smallskipamount

\bibitem{Bi} Bishop, C. J., \emph{Weil-Petersson curves, $\beta$-numbers and minimal surfaces}, \newline http://www.math.stonybrook.edu/~bishop/papers/wpce.pdf 

\bibitem{DeKr} Deift, P., Its, A., and K. Krasovsky, I., \emph{Toeplitz matrices and Toeplitz determinants under the impetus of the Ising model. Some history and some recent results}, Comm. Pure Applied Math. {\bf 66}, no. 9 (2013), 1360--1438.

\bibitem{DuKo} Duits, M., Kozhan, R., \emph{Relative Szeg\H{o} asymptotics for Toeplitz determinants}, Int. Math. Res. Not. IMRN 2019, no. 17, 5441--5496

\bibitem{GM} Garnett, J. B., Marshall, D. E., \emph{Harmonic measure}, New Mathematical Monographs, 2. Cambridge University Press, Cambridge, 2005

\bibitem{Ge} Geronimo, J. S.,  \emph{Szeg\H{o}'s theorem for Hankel determinants}, J. Math. Phys. {\bf 20} (1979), no. 3, 484--491

\bibitem{GrSz} Grenander, U., Szeg\H{o}, G., \emph{Toeplitz forms and their applications}, California Monographs in Mathematical Sciences University of California Press, Berkeley-Los Angeles 1958

\bibitem{Hi} Hirschman, I. I., Jr., \emph{The strong Szeg\H{o} limit theorem for Toeplitz determinants}, Amer. J. Math. {\bf 88} (1966), 577--614

\bibitem{HoJo} Horn, R. A., Johnson, C. R. , \emph{Matrix analysis}, Cambridge University Press, Cambridge, 1985

\bibitem{Jo1} Johansson, K., \emph{On Szeg\H{o}'s asymptotic formula for Toeplitz determinants and generalizations}, Bull. Sci. Math. (2) {\bf 112}, no. 3 (1988), 257--304

\bibitem{JoLa} Johansson, K., Lambert, G., Multivariate normal approximation for traces of random unitary matrices, arXiv:2002.01879

\bibitem{Kr} Krasovsky, I. \emph{Asymptotics for Toeplitz determinants on a circular arc}, arXiv:math/0401256

\bibitem{Mec} Meckes, E. S., \emph{The random matrix theory of the classical compact groups}, Cambridge Tracts in Mathematics, 218. Cambridge University Press, Cambridge, 2019

\bibitem{Po} Pommerenke, C.,  \emph{Univalent functions. With a chapter on quadratic differentials by Gerd Jensen}, Studia Mathematica/Mathematische Lehrbücher, Band XXV. Vandenhoeck \& Ruprecht, Göttingen, 1975

\bibitem{RoWa} Rohde, S., Wang, Y., \emph{The Loewner energy of loops and regularity of driving functions}, Int. Math. Res. Not. IMRN 2021, no. 10, 7715--7763

\bibitem{Si} Simon, B., \emph{Orthogonal polynomials on the unit circle. Parts 1 and 2}, American Mathematical Society Colloquium Publications, 54. American Mathematical Society, Providence, 2005

\bibitem{Sz2}  Szeg\H{o}, G., \emph{\"Uber orthogonale Polynome, die zu einer gegebenen Kurve der komplexen Ebene geh\"oren}, Math. Z. {\bf 9} (1921), 218--270

\bibitem{Sz} Szeg\H{o}, G., \emph{Orthogonal polynomials}, American Mathematical Society Colloquium Publications, Vol. 23 American Mathematical Society, Providence, R.I. 1959

\bibitem{TT} Takhtajan, L. A., Teo, L.-P., \emph{Weil-Petersson metric on the universal Teichm\"uller space}, Mem. Amer. Math. Soc. {\bf 183} (2006), no. 861

\bibitem{Wa} Wang, Y., \emph{Equivalent descriptions of the Loewner energy}, Invent. Math. {\bf 218} (2019), no. 2, 573--621

\bibitem{ViWa1} Viklund, F., Wang, Y, \emph{Interplay between Loewner and Dirichlet energies via conformal welding and flow-lines}, Geom. Funct. Anal. {bf 30} (2020) 289--321

\bibitem{ViWa2} Viklund, F., Wang, Y, \emph{The Loewner-Kufarev energy and foliations by Weil-Petersson quasicircles}, arXiv:2012.0577

\bibitem{Wi} Widom, H., \emph{The strong Szeg\H{o} limit theorem for circular arcs}, Indiana Univ. Math. J. {\bf 21} (1971/72), 277--283



\end{thebibliography}
\end{document}